\documentclass[12pt]{article}
\usepackage{amssymb,amsthm,hyperref,amsmath,bm,amsfonts}
\usepackage{arydshln}
\usepackage{verbatim}
\usepackage{graphicx}
\usepackage{float}

\usepackage{subfigure}
\usepackage{multirow}
\usepackage{subfigure}
\usepackage{color}
\usepackage{appendix}

\newtheorem{theorem}{Theorem}[section]
\newtheorem{lemma}[theorem]{Lemma}
\newtheorem{prop}{Proposition}[section]
\newtheorem{assumption}{Assumption}[section]
\newtheorem{remark}{Remark}[section]
\numberwithin{equation}{section}

\title{Boundary Conditions for Hyperbolic Relaxation Systems with Characteristic Boundaries of Type I}
\author{
Yizhou Zhou\thanks{E-mail: zhouyz16@mails.tsinghua.edu.cn}\\
\small{\textit{Zhou Pei-Yuan Center for Applied Mathematics}},\\
\small{\textit{Tsinghua University, Beijing 100084, China.}}\\[3mm]
Wen-An Yong\thanks{E-mail: wayong@tsinghua.edu.cn}\\
{\small{\textit{Department of Mathematical Sciences}}},\\
{\small{\textit{Tsinghua University, Beijing 100084, China.}}} \\
}

\date{\today}
\begin{document}
\maketitle{}
\begin{abstract}
This work is concerned with boundary conditions
for one-dimensional hyperbolic relaxation systems with characteristic boundaries. 
We assume that the relaxation system satisfies the structural stability condition proposed 
by the second author previously 
and the boundary is characteristic of type I (characteristic for the relaxation system but non-characteristic for the corresponding equilibrium system). 
For this kind of characteristic initial-boundary-value problems, we propose a modified Generalized Kreiss condition (GKC). 
This extends the GKC proposed by 
W.-A. Yong (1999, {\it Indiana Univ. Math. J.} {\bf 48}(1), 115–137) 
for the non-characteristic boundaries to the present characteristic case. 
Under this modified GKC, we derive the reduced boundary condition and verify its validity by combining an energy estimate with the Laplace transform. 
Moreover, we show the existence of boundary-layers for nonlinear problems. 
\end{abstract}

\hspace{-0.5cm}\textbf{Keywords:}
\small{hyperbolic relaxation systems, boundary conditions, characteristic boundaries, generalized Kreiss condition, energy estimate, Laplace transform}\\

\hspace{-0.5cm}\textbf{AMS subject classification:} \small{35L50, 76N20}

\newpage
\section{Introduction}\label{section1} 
\ 
This work is concerned with boundary conditions for hyperbolic relaxation systems
\begin{equation}\label{1.1}
  U_t+\sum_{j=1}^{d}A_j(U)U_{x_j}=\frac{Q(U)}{\epsilon}
\end{equation}
defined for $(x,t)=(x_1,x_2,...,x_d,t)\in ~\Omega\times [0,T)$.  
Here $\Omega$ is a proper subset of $R^d$, $U=U(x,t)\in R^n$ is unknown, $A_j(U)\in R^{n\times n}$ ($j=1,2,...,d$) and $Q(U)\in R^n$ are given smooth functions of $U\in G$ (an open subset of $R^n$ called state space), and $\epsilon$ is a small positive parameter. 
This kind of equations with the small parameter describes various non-equilibrium phenomena. Important examples occur in the kinetic theory \cite{Ga,Le,CFL,Mi}, thermal non-equilibrium flows \cite{Ze}, chemically reactive flows \cite{GM}, compressible viscoelastic flows \cite{Y3,CS}, nonlinear optics \cite{Ha}, non-equilibrium thermodynamics \cite{ZHYY}, traffic flows \cite{BK1,BK2}, and so on.

For such a small parameter problem, a main interest is to understand the limit as $\epsilon$ goes to zero---the so-called zero relaxation limit \cite{CLL,Y1}.
For initial-value problems of \eqref{1.1}, there have been a lot of studies and the interested reader can refer to \cite{Da,Na,Y7} and the references cited therein. Particularly, it was proved in \cite{Y1} that, as $\epsilon$ goes to zero, the solution $U^{\epsilon}$ converges to that of the corresponding equilibrium system under the structural stability condition proposed therein. Here the equilibrium system is a hyperbolic system of first-order partial differential equations and governs the limit. 

On initial-boundary-value problems (IBVPs) of \eqref{1.1}, the literature seems quite rare. The first systematical work is \cite{Y2} where the main issues were well elaborated. In \cite{Y2}, it was observed that the structural stability condition is not enough to ensure the existence of the zero relaxation limit. To attack this problem, the second author of this paper proposed the Generalized Kreiss Condition (GKC) imposed on the boundary conditions for relaxation systems, while the Kreiss condition is for (multi-dimensional) hyperbolic systems without considering the relaxation effects \cite{Kr,GKO}. Under the GKC and the hypothesis that the boundary is non-characteristic for both the relaxation system and its equilibrium system, the reduced boundary condition satisfied by the limit was found and was proved to be well-posed together with the equilibrium system. 
However, the non-characteristic hypothesis required in \cite{Y2} is too restrictive to include many important physical cases \cite{MO,Y4}. 

The purpose of this project is to remove or weaken the non-characteristic hypothesis. Namely, the boundaries may be characteristic for the relaxation system or for its equilibrium system. Thus
there are at least the following possibilities:
\begin{enumerate}
  \item[(I)] The boundary is characteristic for the relaxation system but is non-characteristic for the equilibrium system;
  \item[(II)] The boundary is non-characteristic for the relaxation system but is characteristic for the equilibrium system.
\end{enumerate}
It turns out that the above two cases have different features and difficulties. 
Here we only consider the characteristic boundaries of type (I) and a subsequent work will be devoted to the type (II). 

This paper focuses on the one-dimensional version of \eqref{1.1}
in the half-space $x=x_1>0$. By the boundary $x=0$ to be characteristic for the relaxation system \eqref{1.1}, we mean that the coefficient matrix $A_1(U)$ is not invertible. It is worth to mention that the one-dimensional relaxation problem possesses some key features of usual two-dimensional hyperbolic equations, particularly for the boundary conditions. Such features do not allow to simply decouple the characteristic IBVP into a non-characteristic problem and  ordinary differential equations.

We try to study the above characteristic IBVPs by following the framework in \cite{Y2}. Since the invertibility of $A_1(U)$ plays the role of the starting point in \cite{Y2}, we have to overcome various technical difficulties at the outset and throughout the whole procedure. In particular, the crucial GKC needs to be re-defined. 
The main contributions of this work can be stated as follows.
We propose a modified Generalized Kreiss Condition for the characteristic boundaries of type I.
Under this modified condition, we derive the reduced boundary condition for the corresponding equilibrium system and prove that the resultant IBVP is well-posed. For the linear problem with constant coefficients, we adapt the decomposition method in \cite{GKO} to show the validity of the reduced boundary condition. The method is a combination of the energy estimate and the Laplace transform, which are based on the structural stability condition and the (modified) GKC, respectively.
For nonlinear equations we show the existence of boundary-layers under the modified GKC.

The above results are based on the following assumptions. The system \eqref{1.1} is required to satisfy the structural stability condition in \cite{Y1}.
Moreover, we assume that the symmetrizer $A_0(U)$ in the structural stability condition and the source term $Q(U)$ satisfy 
\begin{equation}\label{1.2}
  A_0(U)Q_U(U)=Q_U^*(U)A_0(U)\quad \text{for}\ U\in \mathcal{E}.
\end{equation}
Here $Q_U(U)$ is the Jacobian matrix of the source term $Q(U)$, the superscript $*$ denotes the transpose of matrices or vectors, and $\mathcal{E}$ stands for the equilibrium manifold $\mathcal{E} = \{U \in G: Q(U) = 0\}$. 
The stability condition and the assumption \eqref{1.2} are quite reasonable and cover many important relaxation models \cite{Y5}. 
In addition, we follow \cite{MO} to assume that the boundary condition does not involve the characteristic modes corresponding to the zero eigenvalue and the boundaries are uniformly characteristic for nonlinear problems. These additional assumptions were shown in \cite{MO} to be crucial for the well-posedness of general hyperbolic IBVPs with characteristic boundaries.

At this point, let us mention that except for \cite{Y2}, the other works about the IBVPs of the relaxation systems seem all for specific models. For example, the one-dimensional linear Jin-Xin model was considered in \cite{XX}. By using the Laplace transform and a matched asymptotic analysis, the authors analyzed the relaxation limit and boundary-layer behaviors. 
In \cite{CH}, the authors studied a specific IBVP for the one-dimensional Kerr-Debye model in nonlinear optics. They justified the zero relaxation limit by exploiting the energy method and the entropy structure of the model.
Recently, a nonlinear discrete-velocity model for traffic flows was investigated in \cite{BK1} and the reduced boundary condition was obtained by solving a boundary Riemann problem. 
The interested reader is referred to \cite{WX,XX2,Xu1,Xu2,ZY} for further works in this direction. 

The rest of this paper is organized as follows. The next section is devoted to linear problems and 
contains four subsections. All assumptions are given in the short Subsection 2.1. 
In Subsection 2.2, we investigate the possible non-existence of the zero relaxation limit and propose a modified GKC for the characteristic boundaries of type I; The reduced boundary condition is derived in Subsection 2.3; Subsection 2.4 is devoted to proving the validity of the reduced boundary condition. In Section 3, we show the existence of boundary-layers for one-dimensional nonlinear problems. Some details of Section 2 are given in Appendix.

For the convenience of the reader, we end this introduction with the aforementioned structural stability condition \cite{Y1} for \eqref{1.1}:
\begin{enumerate}
  \item[(i)] There is an invertible $n\times n$-matrix $P(U)$ and an invertible $r\times r$-matrix $S(U)$ $(0 < r \leq n)$, defined on the equilibrium manifold $\mathcal{E} = \{U \in G:
Q(U) = 0\}$, such that
$$
P(U)Q_U(U)=\left({\begin{array}{*{20}c}
  \vspace{1.5mm}0 & 0 \\
  0 & S(U)
  \end{array}}\right)P(U)\quad \text{for}\  U\in \mathcal{E};
$$
  \item[(ii)] As a hyperbolic system, \eqref{1.1} is symmetrizable, that is, there is a positive
definite symmetric matrix $A_0(U)$ such that
$$
A_0(U)A_1(U) = A_1^*(U)A_0(U) \quad \text{for}\ U\in G;
$$
\item[(iii)] The hyperbolic part and the source term are coupled in the following
sense
$$
A_0(U)Q_U(U)+ Q^*_U(U)A_0(U) \leq -P^*(U)\left({\begin{array}{*{20}c}
  \vspace{1.5mm}0 & 0 \\
  0 & I_r
  \end{array}}\right)
P(U) \quad \text{for}\ U \in \mathcal{E}.
$$
\end{enumerate}
Here and below $I_k$ denotes the unit matrix of order $k$.

\section{Linear Problems}\label{section2}

\subsection{Assumptions}\label{subsection2.1}

We start with one-dimensional linear systems with constant coefficients. Under the structural stability condition, we may as well assume that the linear system has the following form
\begin{eqnarray}\label{2.1}
{\begin{array}{*{20}l}
\vspace{2mm}A_0U_t+ A_1U_x=\dfrac{QU}{\epsilon},\qquad x>0,\ t>0.
  \end{array}}
\end{eqnarray}
Here $A_0$ is a symmetric positive definite matrix, $A_1$ is a symmetric matrix, and $Q=\text{diag}(0,S)$ with $S$ an $r \times r$ ($0\leq r \leq n$) stable matrix.  Moreover, $S$ is a symmetric negative definite matrix under the assumption \eqref{1.2}.
By Theorem 2.2 in \cite{Y1}, the structural stability condition implies that $A_0$ is a block-diagonal matrix of the form
$$
A_0=\left({\begin{array}{*{20}c}
  \vspace{1.5mm}A_{01} &   \\
                       & A_{02}
  \end{array}}\right) 
$$
with the same partition as that of $Q$. Correspondingly, we often write 
$$U=\left({\begin{array}{*{20}c}
  \vspace{1.5mm}u \\
  v
  \end{array}}\right)
,\qquad 
  A_1=\left({\begin{array}{*{20}c}
  \vspace{1.5mm}A_{11} & A_{12} \\
  A_{12}^* & A_{22}
  \end{array}}\right).
$$

Besides the structural stability condition and \eqref{1.2}, the next assumption for \eqref{2.1} is 
\begin{assumption}\label{asp2.1}
The boundary $x=0$ is characteristic for the relaxation system \eqref{2.1}
but is non-characteristic for its equilibrium system
$$
A_{01}u_t+A_{11}u_x=0.
$$  
Namely, $A_{11}$ is invertible but $A_1$ has zero eigenvalues.
\end{assumption}
\noindent With this assumption, we introduce the following three numbers 
\begin{align*}
n_{1}^+=& \text{ the number of positive eigenvalues of} \ A_{11},\\[2mm]
n^+=& \text{ the number of positive eigenvalues of} \ A_1,\\[2mm]
n^o=& \text{ the number of zero eigenvalues of} \ A_1.
\end{align*}

According to the classical theory \cite{BS,GKO} for IBVPs of first-order hyperbolic systems, $n^+$ boundary conditions 
\begin{equation}\label{2.2} 
  BU(0,t)\equiv (B_u,B_v)U(0,t)=b(t)
\end{equation}
should be given at the boundary $x=0$ for \eqref{2.1} and the $n^+\times n$ boundary matrix $B$ should satisfy 
\begin{equation}\label{2.3}
  \det\{BR_A^U\}\neq 0.
\end{equation}
Here $R_A^U$ is an $n\times n^+$ matrix whose columns are the eigenvectors of $A_0^{-1}A_1$ associated with the positive eigenvalues.  
Our last assumption reads as  
\begin{equation}\label{2.4}
  BR_A^0=0
\end{equation}
for each eigenvector $R_A^0$ of $A_0^{-1}A_1$ associated with the zero eigenvalue. The relation \eqref{2.4} means that no characteristic mode corresponding to the zero eigenvalue is involved in the boundary condition above. The necessity of \eqref{2.4} was illustrated in \cite{MO} for the characteristic IBVPs.

In summary, our starting assumptions for the linear relaxation system \eqref{2.1} are the structural stability condition, Assumption \ref{asp2.1}, and those in \eqref{1.2} and \eqref{2.4}.

\subsection{Modified Generalized Kreiss Condition}\label{subsection2.2}
In this subsection, we adapt the standard argument in \cite{Hi,Kr,Y2} to our characteristic IBVPs 
and show that the homogeneous problem
\begin{eqnarray}\label{2.5}
\left\{{\begin{array}{*{20}l}
\vspace{1.5mm}A_0U_t+A_1U_x=\dfrac{QU}{\epsilon},\qquad x>0,\ t>0\\[3mm]
BU(0,t)=0 
  \end{array}}\right.
\end{eqnarray}
with bounded initial data may have exponentially increasing solution for $t>0$ as $\epsilon$ goes to zero. Based on this argument, we propose a modified GKC for the characteristic boundaries of type I.

Following Lemma 3.1 in \cite{Y2}, we consider the following problem 
\begin{eqnarray}\label{2.6}
\left\{{\begin{array}{*{20}l}
\xi A_0\hat{U}+A_1\hat{U}_x=\eta Q\hat{U},\qquad x>0,\\[3mm]
B\hat{U}(0)=0,
  \end{array}}\right.
\end{eqnarray}
where $\eta\geq 0$ and $\xi$ is a complex number with $Re\xi>0$. With $\eta=0$, this reduces to that for IBVPs of conventional hyperbolic equations \cite{Hi}. If \eqref{2.6} 
has a bounded solution $\hat{U}(x)$ for a certain $\eta>0$, then 
$$
U^\epsilon=\exp\bigg(\frac{\xi t}{\eta \epsilon}\bigg)\hat{U}(\frac{x}{\eta \epsilon})\qquad \ \ \ \ 
$$
is the solution to the problem \eqref{2.5} with a bounded initial value. Since $Re\xi>0$ and $\eta>0$, this solution $U^\epsilon$ exponentially increases for $t>0$ as $\epsilon$ goes to zero. 

In order to avoid such exponentially increasing solutions, our first task is to find conditions under what the problem \eqref{2.6} has bounded solutions. 
Before proceeding, we notice that for the characteristic boundary where $A_1$ is not invertible, the differential equation in \eqref{2.6} can not be rewritten as 
$$
\hat{U}_x=A_1^{-1}(\eta Q-\xi A_0)\hat{U}. 
$$
This is quite different from the non-characteristic case. To deal with this difficulty, we formulate the following lemma.
\begin{lemma}\label{lemma2.1}
Under the assumptions in Subsection \ref{subsection2.1}, there exists an $(n-n^o)\times (n-n^o)$-matrix $M(\xi,\eta)$ and an $n^o\times (n-n^o)$-matrix $E(\xi,\eta)$ such that the equation in \eqref{2.6} can be rewritten as
\begin{eqnarray*} 
\left\{{\begin{array}{*{20}l}
\vspace{1.5mm}\hat{V}^I_x=M(\xi,\eta)\hat{V}^I,\\[3mm]
\hat{V}^{II}=E(\xi,\eta)\hat{V}^I.
  \end{array}}\right.
\end{eqnarray*}
Here $\hat{V}=\Phi\hat{U}$ with $\Phi$ an orthonormal matrix given below, $\hat{V}^I$ represents the first $(n-n^o)$ components of $\hat{V}$ and $\hat{V}^{II}$ is the rest components.	
\end{lemma}
 
This lemma allows a very short and simple proof. But here we present a different and tedious proof, in order to study the limit as $\eta$ goes to infinity in the next subsection.

\begin{proof} 
First of all, we recall $Q=\text{diag}(0,S)$ and 
$A_1=\left({\begin{array}{*{20}c}
  A_{11} & A_{12} \\
  A_{12}^* & A_{22}
  \end{array}}\right)$. 
Set $\hat{U}=\left({\begin{array}{*{20}c}
   \hat{u}\\
   \hat{v}
  \end{array}}\right)$ and rewrite the equation in \eqref{2.6} as
\begin{align}
A_{11}\hat{u}_x+A_{12}\hat{v}_x=&-\xi A_{01}\hat{u},\label{2.7}\\[2mm]
A_{12}^*\hat{u}_x+A_{22}\hat{v}_x=&(\eta S-\xi A_{02}) \hat{v}.\label{2.8}
\end{align}
Since $A_{11}$ is invertible, from \eqref{2.7} it follows that 
\begin{equation}\label{2.9}
  \hat{u}_x=-A_{11}^{-1}A_{12}\hat{v}_x-\xi A_{11}^{-1}A_{01}\hat{u}.
\end{equation}
Substituting this relation into \eqref{2.8}, we get
\begin{eqnarray}\label{2.10}
(A_{22}-A_{12}^*A_{11}^{-1}A_{12})\hat{v}_x=(\eta S-\xi A_{02})\hat{v}+\xi A_{12}^*A_{11}^{-1}A_{01}\hat{u}.
\end{eqnarray}

On the the hand, from the congruent transformation
\begin{align}\label{2.11}
\left({\begin{array}{*{20}c}
  \vspace{2mm}I & 0 \\
  -A_{12}^*A_{11}^{-1} & I
  \end{array}}\right)
    \left({\begin{array}{*{20}c}
  \vspace{2mm}A_{11} & A_{12} \\
  A_{12}^* & A_{22}
  \end{array}}\right)
  \left({\begin{array}{*{20}c}
  \vspace{2mm}I & -A_{11}^{-1}A_{12} \\
  0 & I
  \end{array}}\right) 
  =\left({\begin{array}{*{20}c}
  \vspace{2mm}A_{11} & 0 \\
  0 & A_{22}-A_{12}^*A_{11}^{-1}A_{12}
  \end{array}}\right)
\end{align}
for $A_1$, we see that the $r\times r$-matrix $(A_{22}-A_{12}^*A_{11}^{-1}A_{12})$ is symmetric with $n^o$ zero eigenvalues. Then there exists an $r\times r$ orthonormal matrix $\tilde{P}$ satisfying 
\begin{equation}\label{2.12}
  \tilde{P}^*(A_{22}-A_{12}^*A_{11}^{-1}A_{12})\tilde{P}
  =\left({\begin{array}{*{20}c}
  \vspace{1.5mm} \Lambda_2 & 0\\
  0 & 0
  \end{array}}\right).
\end{equation}
Here $\Lambda_2$ is a real $(r-n^o)\times (r-n^o)$ diagonal matrix whose diagonal elements are nonzero eigenvalues of $(A_{22}-A_{12}^*A_{11}^{-1}A_{12})$. 
With the partition of $\text{diag}(\Lambda_2, 0)$, we write $\tilde{P}=(P_2,P_0)$ with $P_2$ an $r\times (r-n^o)$-matrix and $P_0$ an $r\times n^o$-matrix. Then the last equation can be expressed as
\begin{align}
P_2^*(A_{22}-A_{12}^*A_{11}^{-1}A_{12})P_2=&\Lambda_2,\label{2.13}\\[2mm]
(A_{22}-A_{12}^*A_{11}^{-1}A_{12})P_0=&0.\label{2.14}
\end{align}
Moreover, since $\tilde{P}=(P_2,P_0)$ is an orthonormal matrix, we have the decomposition 
\begin{align} 
  \hat{v}=P_2P_2^*\hat{v}+P_0P_0^*\hat{v}:=P_2\hat{w}_2+P_0\hat{w}_0.\label{2.15}
\end{align}

With these preparations, the equation \eqref{2.10} can be rewritten as
\begin{align*} 
  (A_{22}-A_{12}^*A_{11}^{-1}A_{12})P_2\hat{w}_{2x}=(\eta S-\xi A_{02})(P_2\hat{w}_{2}+P_0\hat{w}_{0})+
  \xi A_{12}^*A_{11}^{-1}A_{01}\hat{u}.
\end{align*}
Multiplying this equation with $P_0^*$ and $P_2^*$ from the left and using \eqref{2.13} with \eqref{2.14}, we obtain
\begin{equation*}
  0= P_0^*(\eta S-\xi A_{02})(P_2\hat{w}_{2}+P_0\hat{w}_{0})+
  \xi P_0^*A_{12}^*A_{11}^{-1}A_{01}\hat{u} 
\end{equation*}
and 
\begin{equation*}
  \Lambda_2w_{2x}= P_2^*(\eta S-\xi A_{02})(P_2\hat{w}_{2}+P_0\hat{w}_{0})+
  \xi P_2^*A_{12}^*A_{11}^{-1}A_{01}\hat{u}.
\end{equation*}
Set 
\begin{equation}\label{2.16}
	S_{ij}:=P_i^*(\eta S-\xi A_{02})P_j,\qquad i,j=0,2. 
\end{equation}
The last two equations can be rewritten as 
\begin{equation}\label{2.17}
  S_{00}\hat{w}_0+S_{02} \hat{w}_2+\xi P_0^*A_{12}^*A_{11}^{-1}A_{01}\hat{u}=0
\end{equation}
and
\begin{align} 
  \Lambda_2 \hat{w}_{2x} =& S_{22}\hat{w}_{2}+ S_{20}\hat{w}_{0}+\xi P_2^*
A_{12}^*A_{11}^{-1}A_{01}\hat{u}.\label{2.18}
\end{align}

It is shown in Appendix \ref{AppendA} that the matrix $S_{00}$ is invertible. Then from \eqref{2.17} we can express $\hat{w}_0$ in terms of $\hat{u}$ and $\hat{w}_2$ as 
\begin{equation}\label{2.19} 
  \hat{w}_0=-S_{00}^{-1}(S_{02} \hat{w}_2+\xi P_0^*A_{12}^*A_{11}^{-1}A_{01}\hat{u}).
\end{equation}
Substituting this expression into \eqref{2.18}, we obtain
\begin{align}
        \Lambda_2 \hat{w}_{2x} 
       = &(S_{22}-S_{20}S_{00}^{-1}S_{02})\hat{w}_2 \nonumber \\[2mm]
       & +\xi P_2^* A_{12}^*A_{11}^{-1}A_{01}\hat{u} - \xi S_{20} S_{00}^{-1}P_0^*A_{12}^*A_{11}^{-1}A_{01}\hat{u}.\label{2.20}
\end{align}
Furthermore, with the decomposition in \eqref{2.15}, the equation \eqref{2.9} can be expressed as 
\begin{equation*} 
  \hat{u}_x+A_{11}^{-1}A_{12}(P_2\hat{w}_{2x}+P_0\hat{w}_{0x})=-\xi A_{11}^{-1}A_{01} \hat{u}.
\end{equation*}
With \eqref{2.19}, this equation can be further rewritten as
\begin{align*}
  (I_{n-r}-\xi  A_{11}^{-1}A_{12}P_0 S_{00}^{-1} P_0^*A_{12}^*A_{11}^{-1}A_{01})\hat{u}_x&\\[2mm]
  +A_{11}^{-1}A_{12}(P_2-P_0S_{00}^{-1}S_{02})\hat{w}_{2x}&=-\xi A_{11}^{-1}A_{01} \hat{u}. 
\end{align*}
This and \eqref{2.20} can be written together as 
\begin{align}\label{2.21}
\left({\begin{array}{*{20}c}
  \vspace{1.5mm}I_{n-r}+N_1(\xi,\eta) & N_2(\xi,\eta)\\
  0 & I_{r-n^o}
  \end{array}}\right)
\left({\begin{array}{*{20}c}
  \vspace{1.5mm}\hat{u}\\
  \hat{w}_{2}
  \end{array}}\right)_x= 
\left({\begin{array}{*{20}c}
  \vspace{1.5mm}-\xi A_{11}^{-1}A_{01} & 0\\
  N_3(\xi,\eta) & N_4(\xi,\eta)
  \end{array}}\right)\left({\begin{array}{*{20}c}
  \vspace{1.5mm}\hat{u}\\
  \hat{w}_{2}
  \end{array}}\right),
\end{align}
where 
\begin{eqnarray}\label{2.22} 
\left\{{\begin{array}{*{20}l}
N_1(\xi,\eta) = -\xi (A_{11}^{-1}A_{12}P_0) S_{00}^{-1} (P_0^*A_{12}^*A_{11}^{-1})A_{01},\\[3mm]
N_2(\xi,\eta) = A_{11}^{-1}A_{12}(P_2-P_0S_{00}^{-1}S_{02}),\\[3mm]
N_3(\xi,\eta) = \xi\Lambda_2^{-1}(P_2^*  - S_{20}S_{00}^{-1}P_0^*)A_{12}^*A_{11}^{-1}A_{01} ,\\[3mm]
N_4(\xi,\eta) = \Lambda_2^{-1}[S_{22}-S_{20}S_{00}^{-1}S_{02}].
  \end{array}}\right.
\end{eqnarray}
The matrix $I_{n-r}+N_1(\xi,\eta)$ is proven to be invertible in Appendix \ref{AppendA}.

Now we define
\begin{equation}\label{2.23}
  M(\xi,\eta)=\left({\begin{array}{*{20}c}
  \vspace{1.5mm}I_{n-r}+N_1(\xi,\eta) & N_2(\xi,\eta)\\
  0 & I_{r-n^o}
  \end{array}}\right)^{-1}\left({\begin{array}{*{20}c}
  \vspace{1.5mm}-\xi A_{11}^{-1}A_{01} & 0\\
  N_3(\xi,\eta) & N_4(\xi,\eta)
  \end{array}}\right),
\end{equation}
which is an $(n-n^o)\times (n-n^o)$-matrix, and take
$$
\Phi=\left({\begin{array}{*{20}c}
  \vspace{1.5mm}I_{n-r} & 0 \\
  0  & \tilde{P}^*
  \end{array}}\right)=\left({\begin{array}{*{20}c}
  \vspace{1.5mm}I_{n-r} & 0 \\
  \vspace{1.5mm}0  & P_2^* \\
  0  & P_0^*
  \end{array}}\right).
$$
Thus $\hat{V}^I$ and $\hat{V}^{II}$ in the lemma can be expressed as 
\begin{eqnarray}\label{2.24}
	\hat{V}^I=\left({\begin{array}{*{20}c}
  \vspace{1.5mm}\hat{u}\\
               	P_2^*\hat{v}
  \end{array}}\right)=\left({\begin{array}{*{20}c}
  \vspace{1.5mm}\hat{u}\\
                \hat{w_2}
  \end{array}}\right),\qquad \quad\hat{V}^{II}=P_0^*\hat{v}=\hat{w}_0, 
\end{eqnarray}
and the equation \eqref{2.21} becomes 
$$
\hat{V}^{I}_x=M(\xi,\eta)\hat{V}^I.
$$
Finally, by defining 
\begin{equation}\label{2.25}
	E(\xi,\eta)\equiv \left(E_1(\xi,\eta), \ E_2(\xi,\eta)\right):=\left(-\xi S_{00}^{-1} P_0^*A_{12}^*A_{11}^{-1}A_{01},\ \ -S_{00}^{-1}S_{02} \right),
\end{equation}
the equation \eqref{2.19} can be written as  
$$
\hat{V}^{II}=E(\xi,\eta)\hat{V}^I.
$$ 
This completes the proof.
\end{proof}

For the matrix $M(\xi,\eta)$ defined in \eqref{2.23}, we have the following important fact.

\begin{lemma}\label{lemma2.2}
For any $\eta\geq 0$ and $\xi$ with $Re\xi>0$, the matrix $M(\xi,\eta)$ has no purely imaginary eigenvalue. 
\end{lemma}

We postpone the proof of this lemma in Appendix \ref{AppendB}. With this lemma, we can prove 

\begin{lemma}\label{lemma2.3}
For any $\eta\geq 0$ and $\xi$ with $Re\xi>0$, the matrix $M(\xi,\eta)$ has precisely $n^+$ stable eigenvalues and thereby $(n-n^o-n^+)$ unstable eigenvalues.
\end{lemma}

\begin{proof}
By Lemma \ref{lemma2.2} and the continuity of eigenvalues with respect to the parameters, the numbers of stable and unstable eigenvalues of $M(\xi,\eta)$ are invariant for any $\eta\geq 0$ and $\xi$ with $Re\xi >0$. Thus it suffices to show that the matrix $M(\xi,0)$ has $n^+$ stable eigenvalues. 

When $\eta=0$, it follows from \eqref{2.16} that $S_{ij}=-\xi P_i^*A_{02}P_j$ for $i,j=0,2$. Thus we refer to \eqref{2.22} and have
\begin{align}
I_{n-r}+N_1(\xi,0)=&I_{n-r}+A_{11}^{-1}A_{12}P_0(P_0^*A_{02}P_0)^{-1}P_0^*A_{12}^*A_{11}^{-1}A_{01}\nonumber\\[2mm]
:= & X_1A_{01},\nonumber\\[2mm]
N_4(\xi,0)=&-\xi \Lambda_2^{-1}\big[P_2^*A_{02}P_2 - P_2^*A_{02}P_0(P_0^*A_{02}P_0)^{-1}P_0^*A_{02}P_2\big]\nonumber\\[2mm]
          :=&-\xi \Lambda_2^{-1}X_2 \label{2.26}
\end{align}
with $X_1$, $X_2$ two symmetric matrices. It is clear that $X_1$ is positive definite. On the other hand, it follows from the congruent transformation 
\begin{footnotesize}
\begin{align*}
&\left({\begin{array}{*{20}c}
  \vspace{1.5mm}I_{r-n^o} & -P_2^*A_{02}P_0(P_0^*A_{02}P_0)^{-1}\\
                0 & I_{n^o}
  \end{array}}\right)
\left({\begin{array}{*{20}c}
  \vspace{1.5mm}P_2^*A_{02}P_2 & P_2^*A_{02}P_0\\
                P_0^*A_{02}P_2 & P_0^*A_{02}P_0
  \end{array}}\right)
\left({\begin{array}{*{20}c}
  \vspace{1.5mm}I_{r-n^o} & 0\\
                -(P_0^*A_{02}P_0)^{-1}P_0^*A_{02}P_2 & I_{n^o}
  \end{array}}\right)\\[4mm]
=&\left({\begin{array}{*{20}c}
  \vspace{1.5mm}P_2^*A_{02}P_2 - P_2^*A_{02}P_0(P_0^*A_{02}P_0)^{-1}P_0^*A_{02}P_2 &  \\
                    & P_0^*A_{02}P_0
  \end{array}}\right) \equiv\left({\begin{array}{*{20}c}
  \vspace{1.5mm}X_2 &  \\
                    & P_0^*A_{02}P_0
  \end{array}}\right)
\end{align*}
\end{footnotesize}
that $X_2$ is also positive definite. Moreover, we have
$$
N_3(\xi,0)=\xi \Lambda_2^{-1} N_2^*(\xi,0)A_{01}.
$$ 
Then we can write
\begin{align}
M(\xi,0)=&\xi\left({\begin{array}{*{20}c}
  \vspace{1.5mm}X_1A_{01} & N_2(\xi,0)\\
  0 & I_{r-n^o}
  \end{array}}\right)^{-1}\left({\begin{array}{*{20}c}
  \vspace{1.5mm}-A_{11}^{-1}A_{01} & 0\\
  \Lambda_2^{-1}N_2^*(\xi,0)A_{01} & -\Lambda_2^{-1}X_2
  \end{array}}\right)\nonumber\\[2mm]
  =& \xi \left({\begin{array}{*{20}c}
  \vspace{1.5mm}X_1A_{01} & 0\\
  0 & I_{r-n^o}
  \end{array}}\right)^{-1}\left({\begin{array}{*{20}c}
  \vspace{1.5mm}I_{n-r}  & -N_2(\xi,0)\\
  0 & I_{r-n^o}
  \end{array}}\right) 
    \left({\begin{array}{*{20}c}
  \vspace{1.5mm}-A_{11}^{-1} & 0\\
  0 & -\Lambda_{2}^{-1}
  \end{array}}\right)\nonumber\\[2mm]
  &\ \ \left({\begin{array}{*{20}c}
  \vspace{1.5mm}I_{n-r}  & 0\\
  -N_2^*(\xi,0) & I_{r-n^o}
  \end{array}}\right)\left({\begin{array}{*{20}c}
  \vspace{1.5mm}A_{01}  & 0\\
                   0    & X_2
  \end{array}}\right).\label{2.27}
\end{align}
This indicates that $\frac{1}{\xi}M(\xi,0)$ has the same numbers of positive and negative eigenvalues as 
$$
    \left({\begin{array}{*{20}c}
  \vspace{1.5mm}-A_{11}^{-1} & 0\\
  0 & -\Lambda_{2}^{-1}
  \end{array}}\right)
$$
since $A_{01}$, $X_1$ and $X_2$ are all positive definite.
Hence, $M(\xi,0)$ has $n^+$ stable eigenvalues and the proof is complete.
 \\
\end{proof}

With Lemma \ref{lemma2.3}, we turn to the boundary condition 
$B\hat{U}(0)=0$
in \eqref{2.6}.
In terms of $\hat{V}=\Phi\hat{U}$ given in Lemma \ref{lemma2.1}, the boundary condition can be written as 
\begin{equation}\label{2.28}
	0=B\Phi^*\hat{V}(0)\\[2mm]
 =(B_u,B_v)\left({\begin{array}{*{20}c}
  \vspace{1.5mm}I_{n-r} & 0\\
                 0 & \tilde{P}
  \end{array}}\right)\hat{V}(0).
\end{equation}
Recall that $\tilde{P}=(P_2, P_0)$, where $P_2$ and $P_0$ are the respective $r\times (r-n^o)$ and $r\times n^o$ matrices. And note that  
$$
  \hat{V}=\left({\begin{array}{*{20}c}
  \vspace{1.5mm}\hat{V}^{I}\\
                \hat{V}^{II}
  \end{array}}\right)=
  \left({\begin{array}{*{20}c}
  \vspace{1.5mm}I_{n-n^o}\\
                E(\xi,\eta)
  \end{array}}\right)\hat{V}^{I}=\left({\begin{array}{*{20}c}
  \vspace{1.5mm}I_{n-r} & 0\\
  \vspace{1.5mm}0 & I_{r-n^o}\\
                E_1(\xi,\eta) & E_2(\xi,\eta)
  \end{array}}\right)\hat{V}^I
$$
due to \eqref{2.25} where $E(\xi,\eta)$ is defined. Then the boundary condition \eqref{2.28} can be expressed as
\begin{align}
0=&(B_u,B_v)\left({\begin{array}{*{20}c}
  \vspace{1.5mm}I_{n-r} & 0 & 0\\
  0 & P_2 & P_0
  \end{array}}\right)\left({\begin{array}{*{20}c}
  \vspace{1.5mm}I_{n-r} & 0\\
  \vspace{1.5mm}0 & I_{r-n^o}\\
                E_1(\xi,\eta) & E_2(\xi,\eta)
  \end{array}}\right)\hat{V}^I(0)\nonumber\\[2mm]
  \equiv&(B_u,B_v)\left({\begin{array}{*{20}c}
  \vspace{2mm}I_{n-r} & 0\\
               N_6(\xi,\eta) & N_5(\xi,\eta)
  \end{array}}\right)\hat{V}^I(0)\label{2.29}
\end{align}
with
\begin{eqnarray}\label{2.30} 
\left\{{\begin{array}{*{20}l}
N_5(\xi,\eta)=P_2+P_0E_2(\xi,\eta)=P_2-P_0S_{00}^{-1}S_{02},\\[3mm]
N_6(\xi,\eta)=P_0E_1(\xi,\eta)=-\xi P_0S_{00}^{-1}P_0^*A_{12}^*A_{11}^{-1}A_{01}.
  \end{array}}\right.
\end{eqnarray}

Let $R_M^S(\xi,\eta)$ denote the right-stable matrix of $M(\xi,\eta)$ (see \cite{Y2} for the definition of right-stable matrices). Lemma \ref{lemma2.3} indicates that $R_M^S(\xi,\eta)$ is an $(n-n^o)\times n^+$-matrix. 
If the following $n^+\times n^+$-matrix 
\begin{equation}\label{2.31}
  (B_u,B_v)\left({\begin{array}{*{20}c}
  \vspace{1.5mm}I_{n-r} & 0\\
  N_6(\xi,\eta) & N_5(\xi,\eta)
  \end{array}}\right)R_M^S(\xi,\eta)
\end{equation}
is not invertible for a certain $(\xi,\eta)$ with $\eta>0$, its null space contains a nonzero vector $\zeta$. Thus the nonzero vector $\hat{V}^I(0)=R_M^S(\xi,\eta)\zeta$ satisfies the boundary condition \eqref{2.29}. With this $\hat{V}^I(0)$ as initial value, the ordinary differential equation $\hat{V}^I_x=M(\xi,\eta)\hat{V}^I$ has a bounded solution for $x>0$. In view of Lemma \ref{lemma2.1}, the problem \eqref{2.6} has a bounded solution $\hat{U}$. 

In conclusion, we have finished the task to find the condition for the problem \eqref{2.6} to have a bounded solution. It is that the matrix in \eqref{2.31} is not invertible.

Motivated by the last conclusion, we propose the following\\

\textbf{Modified Generalized Kreiss Condition.} There exists a constant $c_K>0$ such that 
$$\bigg|\det\bigg\{B\left({\begin{array}{*{20}c}
  \vspace{1.5mm}I_{n-r} & 0\\
  N_6(\xi,\eta) & N_5(\xi,\eta)
  \end{array}}\right)R_M^S(\xi,\eta)\bigg\}\bigg|\geq c_K\sqrt{\det\{R_M^{S*}(\xi,\eta)R_M^S(\xi,\eta)\}}$$
for all $\eta\geq 0$ and $\xi$ with $Re \xi>0$. Here $N_5(\xi,\eta)$ and $N_6(\xi,\eta)$ are defined in \eqref{2.30}.\\

About this modified GKC, we have the following interesting remark.
\begin{remark}\label{remark2.1}
In the absence of $S$ or $\eta$, this modified GKC recovers the standard Kreiss condition \eqref{2.3}. 
When $n^o=0$, we recover the GKC in \cite{Y2} for non-characteristic IBVPs.  
We omit the proofs of these conclusions and leave
them for the interested reader. 
\end{remark}

\subsection{Reduced Boundary Condition}\label{subsection2.3}

Under the modified GKC, we derive 
the reduced boundary condition for the equilibrium system 
$$
A_{01}u_t+A_{11}u_x=0
$$
from the relaxation problem \eqref{2.1} with its boundary condition \eqref{2.2}.
Recall that $A_{11}$ has $n_1^+$ positive eigenvalues. According to the classical theory \cite{BS,GKO} for hyperbolic IBVPs, $n_1^+$ linear independent boundary conditions should be prescribed at the boundary $x=0$ for the equilibrium system. 

For this purpose, we follow \cite{Y2} and construct an approximate solution of the form
\begin{align*}
U_\epsilon=\left({\begin{array}{*{20}c}
  \vspace{1.5mm}u_\epsilon\\
                v_\epsilon
  \end{array}}\right)(x,t)=\left({\begin{array}{*{20}c}
  \vspace{1.5mm}\bar{u}\\
                \bar{v}
  \end{array}}\right)(x,t)+\left({\begin{array}{*{20}c}
  \vspace{1.5mm}\mu\\
                \nu
  \end{array}}\right)(x/\epsilon,t).
\end{align*}
As the outer solution, the first part solves 
\begin{gather}
A_{01}\bar{u}_t+A_{11}\bar{u}_x= 0,\qquad \bar{v}= 0, \label{2.32}
\end{gather}
while the boundary-layer correction satisfies  
\begin{align}
\left({\begin{array}{*{20}c}
  \vspace{1.5mm} A_{11} &  A_{12} \\
   A_{12}^* &  A_{22} 
  \end{array}}\right)\left({\begin{array}{*{20}c}
  \vspace{1.5mm}\mu \\
                \nu
  \end{array}}\right)'=\left({\begin{array}{*{20}c}
  \vspace{1.5mm}0 &  \\
                  & S
  \end{array}}\right)\left({\begin{array}{*{20}c}
  \vspace{1.5mm}\mu \\
                \nu
  \end{array}}\right).\label{2.33}
\end{align}
Moreover, the boundary condition \eqref{2.2} should be satisfied:  
\begin{eqnarray}\label{2.34}
	(B_u,B_v)\left({\begin{array}{*{20}c}
  \vspace{1.5mm}\bar{u}+\mu \\
  \bar{v}+\nu
  \end{array}}\right)(0,t)=b(t).
\end{eqnarray}
This together with \eqref{2.32} gives
\begin{equation}\label{2.35}
	B_u\bar{u}(0,t)+(B_u,B_v)\left({\begin{array}{*{20}c}
  \vspace{1.5mm} \mu \\
                 \nu
  \end{array}}\right)(0,t)=b(t).
\end{equation}

Next we turn to the boundary-layer equation \eqref{2.33}, which has the form of the equation in 
\eqref{2.6} with $\xi=0$ and $\eta=1$. Then Lemma \ref{lemma2.1} applies and we have
\begin{align}\label{2.36}
\left({\begin{array}{*{20}c}
  \vspace{1.5mm}I_{n-r}+N_1(0,1) & N_2(0,1)\\
  0 & I_{r-n^o}
  \end{array}}\right)\left({\begin{array}{*{20}c}
  \vspace{1.5mm} \mu \\
                 w_2
  \end{array}}\right)'=& \left({\begin{array}{*{20}c}
  \vspace{1.5mm}0 & 0\\
  N_3(0,1) & N_4(0,1)
  \end{array}}\right)\left({\begin{array}{*{20}c}
  \vspace{1.5mm} \mu \\
                 w_2
  \end{array}}\right)
\end{align}
and
\begin{align}\label{2.37}
w_0= E_1(0,1)\mu +E_2(0,1)w_2.
\end{align}
These equations are analogues of \eqref{2.21} and \eqref{2.24} with \eqref{2.25} in the proof of Lemma \ref{lemma2.1}, with
$$
w_2=P_2^*\nu, \qquad w_0=P_0^*\nu 
$$ 
similar to those in the decomposition \eqref{2.15}. Then we have
\begin{equation}\label{2.38} 
  \nu=P_0w_0+P_2w_2=[P_2-P_0(P_0^*SP_0)^{-1}P_0^*SP_2]w_2\equiv G_1w_2.
\end{equation} 
In this case, \eqref{2.22} becomes
\begin{eqnarray}\label{2.39}
\left\{{\begin{array}{*{20}l}
N_1(0,1) = 0,\\[3mm]
N_2(0,1) = A_{11}^{-1}A_{12}[P_2-P_0(P_0^*SP_0)^{-1}P_0^*SP_2]= A_{11}^{-1}A_{12}G_1,\\[3mm]
N_3(0,1) = 0 ,\\[3mm]
N_4(0,1) = \Lambda_2^{-1}[P_2^*SP_2-P_2^*SP_0(P_0^*SP_0)^{-1}P_0^*SP_2]\equiv \Lambda_2^{-1}G_2.
  \end{array}}\right.
\end{eqnarray}
Thus the equation \eqref{2.36} can be simplified as 
\begin{align}\label{2.40} 
\left({\begin{array}{*{20}c}
  \vspace{1.5mm}I_{n-r}  &  A_{11}^{-1}A_{12}G_1\\
  0 & I_{r-n^o}
  \end{array}}\right)\left({\begin{array}{*{20}c}
  \vspace{1.5mm} \mu \\
                 w_2
  \end{array}}\right)'=&  \left({\begin{array}{*{20}c}
  \vspace{1.5mm}0 & 0\\
  0 & \Lambda_2^{-1}G_2
  \end{array}}\right)\left({\begin{array}{*{20}c}
  \vspace{1.5mm} \mu \\
                 w_2
  \end{array}}\right).
\end{align}
Moreover, \eqref{2.25} becomes
$$
E_1(0,1)=0,\qquad E_2(0,1)=-(P_0^*SP_0)^{-1}P_0^*SP_2.
$$
Then \eqref{2.37} can be written as
\begin{equation}\label{2.41} 
  w_0=-(P_0^*SP_0)^{-1}P_0^*SP_2w_2.
\end{equation}

As boundary-layer corrections, we have $\mu(\infty)=0$ and $w_2(\infty)=P_2^*\nu(\infty)=0$.
Then it follows from the equation \eqref{2.40} that 
\begin{equation}\label{2.42} 
	\mu=-A_{11}^{-1}A_{12}G_1w_2	
\end{equation}
and 
\begin{equation}\label{2.43} 
	w_2'=\Lambda_2^{-1}G_2w_2.
\end{equation}
From the last equation, we see that the initial value for $w_2$ should satisfy     
\begin{equation}\label{2.44}
  L_2^Uw_2(0)=0,
\end{equation}
in order to have a bounded solution $w_2=w_2(x/\epsilon)$. 
Here $L_2^U$ is the left-unstable matrix of $\Lambda_2^{-1} G_2$. 
Recall \eqref{2.12} where $\Lambda_2$ is defined. We notice that $\Lambda_2$ has $(r-n^o-n^++n_1^+)$ negative eigenvalues. Moreover, $G_2$ can be shown to be negative definite with the argument following \eqref{2.26}. Therefore, $L_2^U$ is an $(r-n^o-n^++n_1^+) \times (r-n^o)$-matrix. 

On the other hand, we substitute \eqref{2.38} and \eqref{2.42} into the boundary condition \eqref{2.35} to obtain
\begin{equation}\label{2.45}
   B_u\bar{u}(0,t) +(B_v-B_uA_{11}^{-1}A_{12})G_1w_2(0)=b(t).
\end{equation}
For this, we decompose $\bar{u}(0,t)$ as
\begin{equation}\label{2.46}
\bar{u}(0,t)=R_1^U \alpha + R_1^S \beta. 
\end{equation}
Here $R_1^U$ and $R_1^S$ are right-unstable and right-stable matrices of the invertible matrix $A_{01}^{-1}A_{11}$, respectively. By this decomposition, vectors $\alpha$ and $\beta$ are the respective incoming and outgoing modes for the equilibrium system. Since $A_{11}$ has $n_1^+$ positive eigenvalues and $(n-r-n_1^+)$ negative eigenvalues, the matrices $R_1^U$ and $R_1^S$ are of orders $(n-r)\times n_1^+$ and $(n-r)\times (n-r-n_1^+)$, respectively. With the above decomposition, \eqref{2.45} becomes 
\begin{equation*} 
   B_uR_1^U \alpha +(B_v-B_uA_{11}^{-1}A_{12})G_1w_2(0)=b(t)-B_uR_1^S \beta.
\end{equation*}
Combining this with \eqref{2.44}, we arrive at 
$$ 
  \left({\begin{array}{*{20}c}
  \vspace{1.5mm}B_uR_1^U & (B_v-B_uA_{11}^{-1}A_{12})G_1\\
  0 & L_2^U
  \end{array}}\right)\left({\begin{array}{*{20}c}
  \vspace{1.5mm} \alpha\\
                 w_2(0)
  \end{array}}\right)=
  \left({\begin{array}{*{20}c}
  \vspace{1.5mm} b(t)-B_uR_1^S \beta\\
                 0
  \end{array}}\right). 
$$
This is a system of $(n_1^++r-n^o)$ linear algebraic equations for $(n_1^++r-n^o)$ variables $(\alpha, w_2(0))$, since $B$ has $n^+$ rows, $\alpha$ has $n_1^+$ components, and $w_2(0)$ has $(r-n^o)$ components. 

If the above coefficient matrix is invertible, we can express $\alpha$ and $w_2(0)$ in terms of $\beta$. In this way, we obtain the reduced boundary condition for the equilibrium system in \eqref{2.32} and the initial value $w_2(0)$ for the boundary-layer equation \eqref{2.43}.  With these, we can uniquely determine $w_2$ by solving \eqref{2.43} and $\bar{u}$ by solving \eqref{2.32} with proper initial data. Consequently, the boundary-layer correction $(\mu,\nu)$ and thereby the approximate solution $U_\epsilon$ are constructed uniquely.

To clarify the invertibility, we establish the following analogue of Lemma 3.4 in \cite{Y2}:

\begin{lemma}\label{lemma2.4}
If the boundary matrix $B=(B_u,B_v)$ in \eqref{2.2} satisfies the modified GKC, then the matrix 
$$
\left({\begin{array}{*{20}c}
  \vspace{1.5mm}B_uR_1^U & (B_v-B_uA_{11}^{-1}A_{12})G_1\\
  0 & L_2^U
  \end{array}}\right)
$$
is invertible.
\end{lemma}

\begin{proof}
Like that for Lemma 3.4 in \cite{Y2}, the key of this proof is to find an asymptotic expression of $R_M^S(\xi,\eta)$ in the modified GKC for large $\eta$ and fixed $\xi$. Recall that $R_M^S(\xi,\eta)$ is the right-stable matrix of $M(\xi,\eta)$ and the matrix $M(\xi,\eta)$ is defined in \eqref{2.23}. Thus we need to expand the $N_i(\xi,\eta)$'s defined in \eqref{2.22} into the power series of $1/\eta$.

From \eqref{2.22} and the analytic expansion 
$$
S_{00}^{-1}=\frac{1}{\eta}\left[P_0^*(S- \frac{\xi}{\eta} A_{02})P_0\right]^{-1}=\frac{1}{\eta}(P_0^*SP_0)^{-1}+O(\frac{1}{\eta^2}),
$$
it is not difficult to deduce that 
\begin{align*} 
N_1(\xi,\eta)&=O(\frac{1}{\eta}),\\[3mm]
N_2(\xi,\eta)&=A_{11}^{-1}A_{12}[P_2-P_0(P_0^*SP_0)^{-1}(P_0^* S P_2)]+O(\frac{1}{\eta})\\[3mm]
&= A_{11}^{-1}A_{12}G_1+O(\frac{1}{\eta}),\\[3mm]
N_3(\xi,\eta)&=O(\eta^{0}),\\[3mm]
N_4(\xi,\eta)&=\eta\Lambda_2^{-1}[P_2^* S P_2 - (P_2^* S P_0) (P_0^* S P_0)^{-1} ( P_0^* S P_2)]+O(\eta^0)\\[3mm]
&= \eta\Lambda_2^{-1}G_2+O(\eta^{0}).
\end{align*}
Here $G_1$ and $G_2$ are defined in \eqref{2.38} and \eqref{2.39}. 
Thus we see from \eqref{2.23} that 
\begin{align*}
M(\xi,\eta)
=& \left({\begin{array}{*{20}c}
  \vspace{2mm}I_{n-r}+O(\frac{1}{\eta}) &  A_{11}^{-1}A_{12}G_1+O(\frac{1}{\eta}) \\
  0 & I_{r-n^o}
  \end{array}}\right)^{-1}
  \left({\begin{array}{*{20}c}
  \vspace{2mm}-\xi A_{11}^{-1}A_{01} & 0\\
  O(\eta^0) & \eta\Lambda_2^{-1}G_2+O(\eta^0)
  \end{array}}\right) \\[2mm]
=&\left({\begin{array}{*{20}c}
  \vspace{1.5mm}I_{n-r} & -A_{11}^{-1}A_{12}G_1\\
  0 & I_{r-n^o}
  \end{array}}\right)
  \left({\begin{array}{*{20}c}
  \vspace{1.5mm}-\xi A_{11}^{-1}A_{01} & O(\eta^0)\\
  O(\eta^0) & \eta\Lambda_2^{-1}G_2+O(\eta^0)
  \end{array}}\right)+O(\frac{1}{\eta}).
\end{align*}
Then a simple computation shows
\begin{align}
    &\left({\begin{array}{*{20}c}
  \vspace{1.5mm}I_{n-r} & -A_{11}^{-1}A_{12}G_1\\
  0 & I_{r-n^o}
  \end{array}}\right)^{-1}M(\xi,\eta)\left({\begin{array}{*{20}c}
  \vspace{1.5mm}I_{n-r} & -A_{11}^{-1}A_{12}G_1\\
  0 & I_{r-n^o}
  \end{array}}\right)\nonumber\\[2mm]
  =&~\eta\left({\begin{array}{*{20}c}
  \vspace{1.5mm}0 & 0\\
  0 & \Lambda_2^{-1}G_2
  \end{array}}\right)
  +\left({\begin{array}{*{20}c}
  \vspace{1.5mm}-\xi A_{11}^{-1}A_{01} & O(\eta^0)\\
  O(\eta^0) & O(\eta^0)
  \end{array}}\right)+O(\frac{1}{\eta}) \nonumber\\[2mm]
  \equiv &~\eta \bar{M}(1/\eta). \label{2.47}
\end{align}

Since $\bar{M}(\sigma)$ is analytic and $\Lambda_2^{-1}G_2$ is invertible, we know from \cite{Ka} that there is an invertible matrix $T(\sigma)$, defined in the neighborhood of $\sigma=0$, such that $T(\sigma)$ and $T^{-1}(\sigma)$ are analytic at $\sigma=0$, $T(0)=I_{n-n^o}$, and 
\begin{eqnarray}\label{2.48}
	T^{-1}(\sigma)\bar{M}(\sigma)T(\sigma)=
  \left({\begin{array}{*{20}c}
  \vspace{1.5mm}\bar{M}_1(\sigma) & 0\\
  0 & \bar{M}_2(\sigma)
  \end{array}}\right)
\end{eqnarray}
holds for all sufficiently small $\sigma$. Moreover, it holds that 
$$
\lim_{\eta\rightarrow \infty}\eta\bar{M}_1(1/\eta)=-\xi A_{11}^{-1}A_{01},\qquad \bar{M}_2(0)=\Lambda_2^{-1}G_2. 
$$
Thus $M(\xi,\eta)$ is similar to the block-diagonal matrix $\text{diag}(\eta \bar{M}_1(1/\eta), \eta \bar{M}_2(1/\eta))$ for all sufficiently large $\eta$. 
From the argument following \eqref{2.44} we know that the $(r-n^o)\times (r-n^o)$-matrix $\bar{M}_2(0)$ has $(n^+-n_1^+)$ stable eigenvalues and $(r-n^o-n^++n_1^+)$ unstable eigenvalues. Therefore, for all sufficiently large $\eta$, $\bar{M}_2(1/\eta)$ has the same numbers of stable and unstable eigenvalues as $\bar{M}_2(0)$. Similarly, $\eta\bar{M}_1(1/\eta)$ has $n_1^+$ stable eigenvalues and $(n-r-n_1^+)$ unstable eigenvalues.

Denote by $\bar{R}_1^S(1/\eta)$ and $\bar{R}_2^S(1/\eta)$ the respective right-stable matrices of $\eta\bar{M}_1(1/\eta)$ and $\bar{M}_2(1/\eta)$. According to the previous discussion, they are of orders $(n-r)\times n_1^+$ and $(r-n^o) \times n_2^+$.
By using \eqref{2.47} and \eqref{2.48}, we can directly verify that  
\begin{equation*} 
  R_M^S(\xi,\eta)\\
           =\left({\begin{array}{*{20}c}
  \vspace{1.5mm}I_{n-r} & -A_{11}^{-1}A_{12}G_1\\
  0 & I_{r-n^o}
  \end{array}}\right)T(\frac{1}{\eta})
  \left({\begin{array}{*{20}c}
  \vspace{1.5mm}\bar{R}_1^S(1/\eta) & 0\\
  0 & \bar{R}_2^S(1/\eta)
  \end{array}}\right) 
\end{equation*}
is a right-stable matrix of $M(\xi,\eta)$ when $\eta$ is sufficiently large.

As $\eta\rightarrow \infty$, we get
\begin{align*}
R_M^S\equiv & \lim_{\eta\rightarrow \infty} R_M^S(\xi,\eta)\\
           =& \left({\begin{array}{*{20}c}
  \vspace{1.5mm}I_{n-r} & -A_{11}^{-1}A_{12}G_1\\
  0 & I_{r-n^o}
  \end{array}}\right)
  \left({\begin{array}{*{20}c}
  \vspace{1.5mm}\bar{R}_1^S(0) & 0\\
  0 & \bar{R}_2^S(0)
  \end{array}}\right)\\[2mm]
  =& \left({\begin{array}{*{20}c}
  \vspace{1.5mm}\bar{R}_1^S(0) & -A_{11}^{-1}A_{12}G_1\bar{R}_2^S(0)\\
  0 & \bar{R}_2^S(0)
  \end{array}}\right).
\end{align*}
Here $\bar{R}_1^S(0)$ is a right-stable matrix of $-\xi A_{11}^{-1}A_{01}$ and $\bar{R}_2^S(0)$ is a right-stable matrix of $\Lambda_2^{-1}G_2$. 
Additionally, as $\eta \rightarrow \infty$, the matrices $N_5(\xi,\eta)$ and $N_6(\xi,\eta)$ in \eqref{2.30} have the limits
\begin{align*}
N_6(\xi,\eta)\rightarrow 0,\quad 
N_5(\xi,\eta)\rightarrow G_1.
\end{align*}
Now we let $\eta \rightarrow \infty$ in the modified GKC and find that the matrix
\begin{align*}
  &(B_u,B_vG_1)\left({\begin{array}{*{20}c}
  \vspace{1.5mm}\bar{R}_1^S(0) & -A_{11}^{-1}A_{12}G_1\bar{R}_2^S(0)\\
  0 & \bar{R}_2^S(0)
  \end{array}}\right)
  =\left(
  B_u\bar{R}_1^S(0), \ (B_v-B_uA_{11}^{-1}A_{12})G_1\bar{R}_2^S(0)
  \right)
\end{align*}
is invertible. 
Since $Re\xi>0$, $\bar{R}_1^S(0)$ is actually a right-unstable matrix of $A_{11}^{-1}A_{01}$ and thereby is equal to $R_1^U$ in \eqref{2.46}. Consequently, the matrix 
$$
BR_M^S=\left(
  B_uR_1^U, \ (B_v-B_uA_{11}^{-1}A_{12})G_1R_2^S
  \right)
$$
is invertible, where $R_2^S=\bar{R}_2^S(0)$ is a right-stable matrix of $\Lambda_2^{-1}G_2$.

At last, let $R_2^U$ be a right-unstable matrix of the invertible matrix $\Lambda_2^{-1}G_2$. Then the matrices $(R_2^S,~R_2^U)$ and $L_2^UR_2^U$ are both invertible. Hence, the lemma follows from the simple relation
\begin{align*}
  &\left({\begin{array}{*{20}c}
  \vspace{1.5mm}B_uR_1^U & (B_v-B_uA_{11}^{-1}A_{12})G_1\\
  0  & L_2^U
  \end{array}}\right) 
  \left({\begin{array}{*{20}c}
  \vspace{1.5mm}I_{n_1^+} & 0  & 0 \\
  0  & R_2^S & R_2^U
  \end{array}}\right)\\[4mm]
  =& \left({\begin{array}{*{20}c}
  \vspace{1.5mm}B_uR_1^U & (B_v-B_uA_{11}^{-1}A_{12})G_1R_2^S & (B_v-B_uA_{11}^{-1}A_{12})G_1R_2^U\\
  0  & 0  & L_2^UR_2^U
  \end{array}}\right).
\end{align*}
\end{proof}

Having Lemma \ref{lemma2.4}, we turn to establish the main result of this subsection.

\begin{theorem}\label{thm2.5}
Under the assumptions in Subsection \ref{subsection2.1} and the modified GKC, there exists a full-rank $n_1^+\times n^+$-matrix $B_p$ such that 
the relation 
\begin{equation}\label{2.49}
  B_pB_u\bar{u}(0,t)=B_pb(t),
\end{equation}
as a boundary condition for the equilibrium system, satisfies the Kreiss condition
$$
\det\{B_pB_uR_1^U\}\neq 0 
$$
with $R_1^U$ a right-unstable matrix of $A_{01}^{-1}A_{11}$. 
\end{theorem}

\begin{proof}
Following the proof of Lemma \ref{lemma2.4}, we know that the $n^+\times (n^+-n_1^+)$-matrix $(B_v-B_uA_{11}^{-1}A_{12})G_1R_2^S$ is of full-rank. Then there exists a full-rank $n_1^+ \times n^+$-matrix $B_p$ such that 
\begin{equation}\label{2.50}
  B_p (B_v-B_uA_{11}^{-1}A_{12})G_1R_2^S=0.
\end{equation}
Notice that $R_2^S$ is independent of $\xi$, so is $B_p$. 
Thus it follows from the invertibility of $BR_M^S$ and 
\begin{align*}
B_pBR_M^S=&B_p(B_uR_1^U,\ (B_v-B_uA_{11}^{-1}A_{12})G_1R_2^S)\\[2mm]
  =&(B_pB_uR_1^U, 0) 
\end{align*}
that the square matrix $B_pB_uR_1^U$ is invertible. This completes the proof.
\end{proof}

We end this subsection with the following remark.

\begin{remark}
The boundary condition \eqref{2.49} is derived from \eqref{2.45} by multiplying $B_p$ from the left and by using \eqref{2.50}.  
Here we also use \eqref{2.44} which implies that $w_2(0)\in \text{span}\{R_2^S\}$.   
\end{remark}
\noindent In the next subsection, we will show that the boundary condition \eqref{2.49} is satisfied by the relaxation limit. In this sense, it is referred to as the reduced boundary condition.

\subsection{Validity of the reduced boundary condition}\label{subsection2.4} 

In this subsection, we show that the relaxation limit of smooth solution $U^\epsilon$ to the relaxation system \eqref{2.1} with boundary condition \eqref{2.2} satisfies the boundary condition \eqref{2.49} as well as the equilibrium system \eqref{2.32}. This will be done
by estimating $(U^\epsilon-U_\epsilon)$ with the decomposition method \cite{GKO} mainly based on the modified GKC. Here $U_\epsilon$ is the approximate solution constructed in the previous subsection.

Firstly, we make some assumptions on the initial and boundary data to avoid unnecessary technical difficulties. 
The initial value $U_0(x)$ for the relaxation system \eqref{2.1} is assumed to be in equilibrium:
\begin{equation}\label{2.53}
	U_0(x)=\left({\begin{array}{*{20}c}
  \vspace{1.5mm}u_0(x) \\
                0
\end{array}}\right).
\end{equation}
Here $u_0=u_0(x)$ represents the first $(n-r)$ components of $U_0(x)$. Moreover, we assume that the initial value and the boundary data in \eqref{2.2} are compatible at the corner $(x,t)=(0,0)$:
\begin{equation}\label{2.54}
  BU_0(0)=b(0).
\end{equation} 
These imply 
\begin{equation*}
  B_pB_uu_0(0)=B_pb(0),
\end{equation*}
meaning that the boundary condition \eqref{2.49} is compatible with the initial value $u_0$ for the equilibrium system \eqref{2.32}.

Our main result of this subsection can be stated as  

\begin{theorem}\label{thm2.6}
Under the assumptions in Subsection \ref{subsection2.1} and the modified GKC, let $U_0\in H^1(R^+)$ and $b(t)\in H^1(0,T)$ satisfy \eqref{2.53} and \eqref{2.54}. Then there exists a constant $K>0$ such that the following error estimate 
\begin{align*}
\|(U^\epsilon-U_\epsilon)(\cdot,t)\|_{L^2(R^+)}\leq K \epsilon^{1/2} 
\end{align*}
holds for all time $t\in[0,T]$.
\end{theorem}

As the first step to prove this theorem, we derive equations for the difference $W\equiv U^\epsilon-U_\epsilon$. 
From \eqref{2.32} and \eqref{2.33} we deduce that the approximate solution
\begin{align*}
U_\epsilon=\left({\begin{array}{*{20}c}
  \vspace{1.5mm}u_\epsilon \\
  v_\epsilon
\end{array}}\right) =&
\left({\begin{array}{*{20}c}
  \vspace{1.5mm}\bar{u} \\
  \bar{v} 
\end{array}}\right) +
\left({\begin{array}{*{20}c}
  \vspace{1.5mm}\mu \\
  \nu 
\end{array}}\right)
\end{align*}
satisfies
\begin{align}
&\left({\begin{array}{*{20}c}
  \vspace{1.5mm}A_{01} & 0 \\
  0 & A_{02}
  \end{array}}\right)\left({\begin{array}{*{20}c}
  \vspace{1.5mm}u_\epsilon \\
  v_\epsilon
  \end{array}}\right)_t+
\left({\begin{array}{*{20}c}
  \vspace{1.5mm}A_{11} & A_{12} \\
  A_{12}^* & A_{22}
  \end{array}}\right)
\left({\begin{array}{*{20}c}
  \vspace{1.5mm}u_\epsilon \\
  v_\epsilon
  \end{array}}\right)_x-\frac{1}{\epsilon}
\left({\begin{array}{*{20}c}
  \vspace{1.5mm}0 & 0 \\
  0 & S
  \end{array}}\right)
\left({\begin{array}{*{20}c}
  \vspace{1.5mm}u_\epsilon \\
  v_\epsilon
  \end{array}}\right)\nonumber\\[2mm]
=&\left({\begin{array}{*{20}c}
  \vspace{1.5mm}0 \\
  A_{12}^*\bar{u}_{x}
  \end{array}}\right)+
\left({\begin{array}{*{20}c}
  \vspace{1.5mm}A_{01}\mu_t  \\
  A_{02}\nu_t
\end{array}}\right)  
:= \left({\begin{array}{*{20}c}
  \vspace{1.5mm}0 \\
  F_1
\end{array}}\right)+
F_2.\label{2.55}
\end{align}
Thus the difference $W$ solves
\begin{align}\label{2.56}
A_0W_t+
A_1
W_x=\frac{1}{\epsilon}
\left({\begin{array}{*{20}c}
  \vspace{1.5mm}0 & 0 \\
  0 & S
  \end{array}}\right)W-\left({\begin{array}{*{20}c}
  \vspace{1.5mm}0 \\
  F_1
\end{array}}\right)-F_2.
\end{align}
Moreover, it follows from \eqref{2.34} and \eqref{2.2} that $W$ satisfies the boundary condition 
\begin{eqnarray}\label{2.57}
  BW(0,t)=0.  
\end{eqnarray} 
On the other hand, with $\bar{u}(x,0)=u_0(x)$ we can deduce from the compatibility assumption \eqref{2.54} and the unique solvability of the boundary-layer correction $(\mu,\nu)$ claimed preceding Lemma \ref{lemma2.4} that 
$$
\mu|_{t=0}\equiv 0,\quad \nu|_{t=0}\equiv 0.
$$
Therefore, $W$ vanishes at $t=0$:
\begin{align}\label{2.58}
W(x,0)=0. 
\end{align}
In this way, we derive an IBVP \eqref{2.56}-\eqref{2.58} for $W$.

To estimate $W$, we follow \cite{GKO} and make the following decomposition
$$
W=W_1+W_2.
$$
Here $W_1$ solves 
\begin{align}\label{2.59}
\left\{
{\begin{array}{*{20}l}
  \vspace{2mm}A_0W_{1t}+A_1W_{1x}=\dfrac{1}{\epsilon}
QW_1-\left({\begin{array}{*{20}c}
  \vspace{2mm}0 \\
  F_1
\end{array}}\right)-F_2, \\[2mm]
L_+A_0^{1/2} W_1(0,t)=0,\\[3mm]
W_1|_{t=0}=0,
\end{array}}
\right.
\end{align}
and $W_2$ satisfies
\begin{align}\label{2.60}
\left\{
{\begin{array}{*{20}l}
  \vspace{2mm}A_0W_{2t}+
A_1W_{2x}=\dfrac{1}{\epsilon}
QW_2, \\[2mm]
BW_2(0,t)=-BW_1(0,t),\\[3mm]
W_2|_{t=0}=0.
\end{array}}
\right.\qquad \qquad \qquad \quad  
\end{align}
In \eqref{2.59}, $L_+$ is an $n^+\times n$-matrix consisting of the $n^+$ left-eigenvectors of matrix $A_0^{-1/2}A_1A_0^{-1/2}$ associated with its positive eigenvalues. 
Note that $A_0^{-1/2}A_1A_0^{-1/2}$ is symmetric. Then $L_+$ can be the first 
$n^+$ rows of the following orthonormal matrix $L$ satisfying
\begin{equation}\label{2.61}
	L L^*=I_{n},\qquad LA_0^{-1/2}A_1A_0^{-1/2}L^*=\text{diag}(\Lambda_+,\Lambda_-,0).
\end{equation}
Here $\Lambda_+$ and $\Lambda_-$ are diagonal matrices whose entries are the $n^+$ positive and $(n-n^o-n^+)$ negative eigenvalues of $A_0^{-1/2}A_1A_0^{-1/2}$, respectively. Corresponding to the partition $\text{diag}(\Lambda_+,\Lambda_-,0)$, we can write 
$$
L=\left({\begin{array}{*{20}c}
  				L_+\\
                L_-\\
                L_0
  \end{array}}\right),
$$ 
where $L_-$ and $L_0$ are $(n-n^o-n^+)\times n$ and $n^o\times n$-matrices, respectively. 

For IBVPs \eqref{2.59} and \eqref{2.60}, we have the following conclusions.

\begin{lemma}\label{lemma2.7}
If $F_1,F_2\in L^2([0,T]\times R^+)$, then the IBVP \eqref{2.59} has a unique solution $W_1=W_1(x,t)$ satisfying 
\begin{align}
&\max_{t\in[0,T]}\|W_1(\cdot,t)\|_{L^2(R^+)}^2+ \int_0^T|L_-A_0^{1/2}W_1(0,t)|^2dt \nonumber\\[2mm]
\leq &\ C(T)\bigg(\epsilon\|F_1\|^2_{L^2([0,T]\times R^+)} + \|F_{2}\|^2_{L^2([0,T]\times R^+)}\bigg).\label{2.62} 
\end{align}
Here $C(T)$ is a generic constant depending only on $T$.
\end{lemma}

\begin{lemma}\label{lemma2.8}
The IBVP \eqref{2.60} has a unique solution $W_2=W_2(x,t)$ satisfying 
\begin{align}
&\max_{t\in[0,T]}\|W_2(\cdot,t)\|_{L^2(R^+)}^2+\int_0^T|W_2(0,t)|^2dt \leq   C(T) \int_0^T|L_-A_0^{1/2}W_1(0,t)|^2 dt . \label{2.63}
\end{align}
\end{lemma}

Assuming these two lemmas for the moment, we present \\

\textbf{A Proof of Theorem \ref{thm2.6}:}
By Theorem \ref{thm2.5}, we know that the equilibrium system \eqref{2.32} with the boundary condition \eqref{2.49} and initial condition $\bar{u}(x,0)=u_0(x)$ constitutes a well-posed IBVP. 
According to the classical theory for hyperbolic IBVPs (see, e.g., \cite{BS}), the well-posed IBVP has a unique solution $\bar{u}\in C([0,T];H^1(R^+))$. Then we have 
\begin{gather*}
\|F_1\|^2_{L^2([0,T]\times R^+)} \equiv \|A_{12}^*\bar{u}_x\|^2_{L^2([0,T]\times R^+)}  \leq C(T) .
\end{gather*}
Moreover, we have 
\begin{align*}
\|F_{2}\|^2_{L^2([0,T]\times R^+)}\equiv &  \int_{R^+}\int_0^T\left[\left|A_{01}\mu_{t}(\frac{x}{\epsilon},t)\right|^2+\left|A_{02}\nu_{t}(\frac{x}{\epsilon},t)\right|^2\right]dtdx \\[2mm]
=& \epsilon \int_{R^+}\int_0^T \left[|A_{01}\mu_{t}(y,t)|^2+|A_{02}\nu_{t}(y,t)|^2\right]dtdy  
\leq C(T) \epsilon.
\end{align*}
With these, Theorem \ref{thm2.6} immediately follows from a simple linear combination of the two estimates in \eqref{2.62} and \eqref{2.63}. Hence Theorem \ref{thm2.6} is proved.\\

It remains to prove Lemmas \ref{lemma2.7} and \ref{lemma2.8}.

\begin{proof}(Lemma \ref{lemma2.7})
We firstly claim that the boundary condition in \eqref{2.59} satisfies the Kreiss condition. In fact, since $A_0^{-1/2}A_1A_0^{-1/2}L_+^*=L_+^*\Lambda_+$ due to \eqref{2.61}, $A_0^{-1/2}L_+^*$ is a right-unstable matrix of $A_0^{-1}A_1$. For this right-unstable matrix, we have
$$
\det\{(L_+A_0^{1/2})(A_0^{-1/2}L_+^*)\} = 1,
$$
namely, the Kreiss condition is satisfied. Thus the existence follows from the existence theory (Theorem 1.12 in \cite{MO}) for hyperbolic IBVPs with uniformly characteristic boundaries.

For the estimate \eqref{2.62}, we multiply the differential equation in \eqref{2.59} with  
$W_1^*$ from the left to obtain
\begin{align*}
W_1^*A_0W_{1t}+W_1^*A_1W_{1x}=&\dfrac{1}{\epsilon}
W_1^*\left({\begin{array}{*{20}c}
  \vspace{1.5mm}0 & 0 \\
  0 & S
  \end{array}}\right)W_1-W_1^{II*}F_1-W_1^*F_{2}.
\end{align*}
Here $W_1^{II}$ represents the last $r$ components of $W_1$. From this equation, we can easily derive 
\begin{align*}
\frac{d(W_1^*A_0W_1)}{dt}+(W_1^*A_1W_1)_x
\leq & -\frac{c_0}{\epsilon}|W_1^{II}|^2+ 2|ReW_1^{II*}F_1| + 2|ReW_1^*F_{2}|\\[2mm]
\leq & -\frac{c_0}{2\epsilon}|W_1^{II}|^2+  \frac{2\epsilon}{c_0}|F_1|^2 + |W_1|^2+|F_{2}|^2
\end{align*}
with $c_0>0$ a constant, since $A_0$ and $A_1$ are symmetric and $S$ is negative definite.
Integrating the last inequality over $x\in[0,\infty)$ yields
\begin{align}\label{2.64}
&\frac{d}{dt}\left(\int_{R^+}W_1^*A_0W_1dx \right) - W_1^*(0,t)A_1W_1(0,t)  
\leq  \frac{2\epsilon}{c_0} \|F_1\|^2_{L^2( R^+)} + \|W_1\|^2_{L^2( R^+)}+\|F_{2}\|^2_{L^2( R^+)}. 
\end{align}

Thanks to the boundary condition $L_+A_0^{1/2}W_1(0,t)=0$, it follows from \eqref{2.61} that
\begin{align}
-W_1^*(0,t)A_1W_1(0,t)=&-W_1^*(0,t)A_0^{1/2}(A_0^{-1/2}A_1A_0^{-1/2})A_0^{1/2}W_1(0,t)\nonumber \\[2mm]
=& -W_1^*(0,t)A_0^{1/2}(L_+^*\Lambda_+L_++L_-^*\Lambda_-L_-)A_0^{1/2}W_1(0,t)\nonumber\\[2mm]
                 =& -W_1^*(0,t)A_0^{1/2} L_-^*\Lambda_-L_- A_0^{1/2}W_1(0,t)\nonumber\\[2mm]
                 \geq & \ c_1|L_-A_0^{1/2}W_1(0,t)|^2 \label{2.65}
\end{align}
with $c_1>0$ a constant.
Since $W_1|_{t=0}=0$ and
$$
C^{-1} \|W_1\|^2_{L^2(R^+)} \leq \int_{R^+}W_1^*A_0W_1dx\leq C \|W_1\|^2_{L^2(R^+)} 
$$
with $C$ a generic constant,
we can use Gronwall's inequality in \eqref{2.64} to get
\begin{align*}
 \max_{t\in[0,T]}\|W_1(t)\|_{L^2( R^+)}^2  
\leq & ~ Ce^{CT}\bigg( \epsilon\|F_1\|^2_{L^2([0,T] \times R^+)}+\|F_{2}\|^2_{L^2([0,T] \times R^+)}   \bigg). 
\end{align*}
Having this and \eqref{2.65}, we integrate \eqref{2.64} to obtain
\begin{align*} 
\int_0^T|L_-A_0^{1/2}W_1(0,t)|^2dt\leq C(T)\bigg( \epsilon\|F_1\|^2_{L^2([0,T] \times R^+)}+\|F_{2}\|^2_{L^2([0,T] \times R^+)}\bigg).
\end{align*}
The last two inequalities together give \eqref{2.62} and the proof is complete.
\end{proof}

Finally, we present a proof of Lemma \ref{lemma2.8}.

\begin{proof}(Lemma \ref{lemma2.8})
Since the modified GKC implies the standard Kreiss condition (see Remark \ref{remark2.1}), the IBVP \eqref{2.60} has a unique $L^2([0,T]\times R^+)$-solution $W_2=W_2(x,t)$ by the aforementioned existence theory in \cite{MO}. 

The estimate \eqref{2.63} can be derived with the following three steps.

\textbf{Step 1:} 
Multiplying the equation in \eqref{2.60} with $W_2^*$ from the left yields
\begin{align*}
\frac{d}{dt}(W_2^*A_0W_2)+(W_2^*A_1W_{2})_x= \frac{2}{\epsilon}W_2^*QW_2\leq 0.
\end{align*}
Notice that the initial value of $W_2$ is zero. We integrate the last inequality over $x\in [0,+\infty)$ and $t\in [0,T]$ to obtain
\begin{align*}
\int_{R^+}W_2^*(x,t)A_0W_2(x,t) dx\leq & \int_0^TW_2^*(0,t)A_1W_2(0,t)dt\leq C\int_0^T|W_2(0,t)|^2dt,\qquad \forall\ t\in[0,T].
\end{align*}
Since $A_0$ is a positive definite matrix, we have 
\begin{equation*} 
  \|W_2(\cdot,t)\|^2_{L^2(R^+)}\leq C\int_0^{T} |W_2(0,t)|^2dt,\qquad \forall\ t\in[0,T]. 
\end{equation*}
With this, it remains to bound $\int_0^{T} |W_2(0,t)|^2dt$ in terms of $\int_0^T|L_-A_0^{1/2}W_1(0,t)|^2 dt$.

\textbf{Step 2:} As a preparation for Step 3, we show that $g(t):=-BW_1(0,t)$ can be expressed in terms of $L_-A_0^{1/2}W_1(0,t)$.
Recall that $L$ is an orthonormal matrix in \eqref{2.61}. Then we have
\begin{align*}
I_n=&\left({\begin{array}{*{20}c}
          \vspace{1.5mm}L_+A_0^{1/2}\\
                \vspace{1.5mm}L_-A_0^{1/2}\\
                L_0A_0^{1/2}
  \end{array}}\right)(A_0^{-1/2}L_+^*,~A_0^{-1/2}L_-^*,~A_0^{-1/2}L_0^*)\\[2mm]
  =&A_0^{-1/2}L_+^*L_+A_0^{1/2} + A_0^{-1/2}L_-^*L_-A_0^{1/2} + A_0^{-1/2}L_0^*L_0A_0^{1/2}. 
\end{align*}
Using this and the boundary condition in \eqref{2.59}, we rewrite $g(t)$ as 
\begin{align*}
g(t)=&-BA_0^{-1/2}L_+^*L_+A_0^{1/2}W_1(0,t)-BA_0^{-1/2}L_-^*L_-A_0^{1/2}W_1(0,t)-BA_0^{-1/2}L_0^*L_0A_0^{1/2}W_1(0,t)\\[2mm]
=&-BA_0^{-1/2}L_-^*L_-A_0^{1/2}W_1(0,t)-BA_0^{-1/2}L_0^*L_0A_0^{1/2}W_1(0,t).
\end{align*}
On the other hand, from \eqref{2.61} we see that $A_0^{-1/2}A_1A_0^{-1/2}L_0^*=0$. This means that each column of $A_0^{-1/2}L_0^*$ is an eigenvector of $A_0^{-1}A_1$ associate with the zero eigenvalue. By the assumption \eqref{2.4}, we have
$$
BA_0^{-1/2}L_0^*=0.
$$
Consequently, we obtain
\begin{align}\label{2.75}
g(t)=&-BA_0^{-1/2}L_-^*L_-A_0^{1/2}W_1(0,t).
\end{align}

\textbf{Step 3:} 
For the $L^2$-solution $W_2=W_2(x,t)$, define its Laplace transform with respect to $t$:
$$
\hat{W_2}(x,\xi)=\int_0^{\infty}e^{-\xi t}W_2(x,t)dt,\qquad Re \xi>0.
$$
Then we deduce from the IBVP \eqref{2.60} that 
\begin{align}\label{2.68}
\left\{
{\begin{array}{*{20}l}
\vspace{2mm}A_1\hat{W}_{2x}=(\eta Q-\xi A_0)\hat{W}_2,\\
\vspace{2mm} B\hat{W}_2(0,\xi)=\hat{g}(\xi),\\
\|\hat{W}_2(\cdot,\xi)\|_{L^2}<\infty~~ \text{for ~a.e.} ~~\xi.
\end{array}}
\right.
\end{align}
Here we use the notation $\eta=1/\epsilon$. 

Observe that the first line in \eqref{2.68} is just the equation in \eqref{2.6}. Then Lemma \ref{lemma2.1} applies and gives 
\begin{eqnarray}\label{2.69}
\left\{{\begin{array}{*{20}l}
\vspace{1.5mm} \hat{V}^I_{2x}=sM(\xi',\eta')\hat{V}_2^I,\\[3mm]
\hat{V}_2^{II}=E(\xi',\eta')\hat{V}_2^I.
  \end{array}}\right.
\end{eqnarray}
Here $s=\sqrt{\eta^2+|\xi|^2}$, $\xi'=\xi/s$ and $\eta'=\eta/s$, $\hat{V}_2^I$ denotes the first $(n-n^o)$ components of $\Phi\hat{W}_2$, and $\hat{V}_2^I$ denotes the other $n^o$ components of $\Phi\hat{W}_2$.
In terms of $\hat{V}_2= \Phi\hat{W}_2$, the boundary condition in \eqref{2.68} becomes (see \eqref{2.28}-\eqref{2.30})
\begin{align}\label{2.70}
(B_u,B_v)\left({\begin{array}{*{20}c}
  \vspace{1.5mm}I_{n-r} & 0\\
  N_6(\xi',\eta') & N_5(\xi',\eta')
  \end{array}}\right)\hat{V}_2^I(0,\xi)=\hat{g}(\xi).
\end{align}

Let $R_M^S=R_M^S(\xi',\eta')$ and $R_M^U=R_M^U(\xi',\eta')$ be the right-stable and right-unstable matrices of $M=M(\xi',\eta')$:
\begin{align*}
MR_M^S=R_M^SM^S,\qquad\qquad 
MR_M^U=R_M^UM^U
\end{align*}
with $M^S$ a stable-matrix and $M^U$ an unstable-matrix. In view of the Schur decomposition, we may assume
$$
(R_M^S,\ R_M^U)\left({\begin{array}{*{20}c}
  \vspace{1.5mm} R_M^{S*} \\
                 R_M^{U*}
\end{array}}\right)=\left({\begin{array}{*{20}c}
  \vspace{1.5mm} R_M^{S*} \\
                 R_M^{U*}
\end{array}}\right)(R_M^S,\ R_M^U)=I_{n-n^o}.
$$
Thus, we deduce from \eqref{2.69} that
\begin{align*}
\left({\begin{array}{*{20}c}
  \vspace{1.5mm} R_M^{S*} \\
                 R_M^{U*}
\end{array}}\right)\hat{V}^I_{2x}=s
  \left({\begin{array}{*{20}c}
  \vspace{1.5mm}M^S & \\
                    & M^U
\end{array}}\right)\left({\begin{array}{*{20}c}
  \vspace{1.5mm} R_M^{S*} \\
                 R_M^{U*}
\end{array}}\right)\hat{V}_2^I.
\end{align*}
Since $\|\hat{V_2}(\cdot,\xi)\|_{L^2}=\|\Phi\hat{W_2}(\cdot,\xi)\|_{L^2}<\infty$ for a.e. $\xi$ and $M^U$ is an unstable-matrix, it must be
\begin{equation}\label{2.71}
   R_M^{U*}\hat{V}_2^I=0.
\end{equation}
Therefore, the boundary condition \eqref{2.70} becomes
\begin{align*}
(B_u,B_v)\left({\begin{array}{*{20}c}
  \vspace{1.5mm}I_{n-r} & 0\\
  N_6(\xi',\eta') & N_5(\xi',\eta')
  \end{array}}\right)R_M^{S}R_M^{S*}\hat{V}_2^I(0,\xi)=\hat{g}(\xi).
\end{align*}
The last equation leads to 
\begin{equation}\label{2.72}
  |R_M^{S*}\hat{V}_2^I(0,\xi)|^2\leq C|\hat{g}(\xi)|^2,
\end{equation}
since the matrix 
$$
\bigg[(B_u,B_v)\left({\begin{array}{*{20}c}
  \vspace{1.5mm}I_{n-r} & 0\\
  N_6(\xi',\eta') & N_5(\xi',\eta')
  \end{array}}\right)R_M^{S}\bigg]^{-1}
$$ 
is uniformly bounded due to the modified GKC. 

It is shown in Appendix \ref{appendC} that $E(\xi',\eta')$ and $R_M^S(\xi',\eta')$ are uniformly bounded for $|\xi'|\leq 1$ and $0 \leq \eta'\leq 1$. Then we deduce from \eqref{2.69}, \eqref{2.71} and \eqref{2.72} that 
\begin{align*}
|\hat{W}_2(0,\xi)|^2 =&|\Phi^*\hat{V}_2(0,\xi)|^2
\leq C|\hat{V}_2^I(0,\xi)|^2 
= C|R_M^SR_M^{S*}\hat{V}_2^I(0,\xi)|^2 
\leq  C|\hat{g}(\xi)|^2.
\end{align*}
Applying Parseval's identity to this inequality yields
\begin{align}\label{2.73}
\int_0^{\infty}e^{-2t Re\xi }|W_2(0,t)|^2dt \leq & C \int_0^{\infty}e^{-2t Re\xi }|g(t)|^2dt\leq C \int_0^{\infty} |g(t)|^2dt ,\qquad Re\xi>0.
\end{align}
Because the right-hand side is independent of $Re\xi$, we have 
$$
\int_0^{\infty} |W_2(0,t)|^2dt \leq C \int_0^{\infty} |g(t)|^2dt.
$$

By using the trick from \cite{GKO}, the integral interval $[0,\infty)$ in the last inequality can be easily changed to $[0,T]$. Indeed, according to the classical theory of first-order hyperbolic system, the solution at any finite time $t = T$ does not depend on the boundary data $g(t)$
for $t > T$. Therefore, with the following boundary data  
$$
\tilde{g}(t)=\left\{{\begin{array}{*{20}l}
              g(t),\qquad &t\in[0,T], \\[2mm]
                 0,\qquad &t>T,
\end{array}}\right.
$$
the new solution $\tilde{W}_2=\tilde{W}_2(x,t)$ satisfies  
$$
\int_0^{\infty} |\tilde{W}_2(0,t)|^2dt \leq C \int_0^{\infty} |\tilde{g}(t)|^2dt.
$$
Since $W_2=\tilde{W}_2$ for $t \leq T$, we have
\begin{align*} 
\int_0^{T} |W_2(0,t)|^2dt \leq \int_0^{\infty} |\tilde{W}_2(0,t)|^2dt \leq C\int_0^{\infty} |\tilde{g}(t)|^2dt=C\int_0^{T} |g(t)|^2dt. 
\end{align*}
This together with \eqref{2.75} gives 
$$
\int_0^{T} |W_2(0,t)|^2dt\leq C\int_0^{T} |g(t)|^2dt  \leq C\int_0^T|L_-A_0^{1/2}W_1(0,t)|^2 dt 
$$
and hence the proof is complete.
\end{proof}

\section{One-dimensional nonlinear problems}\label{section3}

\subsection{Assumptions}\label{subsection3.1}
Now we consider the one-dimensional nonlinear system
\begin{eqnarray}\label{3.1}
U_t+F(U)_x=Q(U)/\epsilon 
\end{eqnarray}
in $0\leq x,t<\infty$ with the boundary condition 
\begin{equation}\label{3.2}
	B(U(0,t);t)=0.
\end{equation}
Here 
$$
U=\left({\begin{array}{*{20}c}
\vspace{1.2mm} u\\[1mm] 
               v
  \end{array}}\right),\qquad F(U)=\left({\begin{array}{*{20}c}
\vspace{1.2mm} f(u,v)\\[1mm] 
               g(u,v)
  \end{array}}\right),\qquad Q(U)=\left({\begin{array}{*{20}c}
\vspace{1.2mm} 0\\ 
               q(u,v)
  \end{array}}\right),
$$
 where $u\in R^{n-r}$ and $v\in R^r$ are unknowns, $f=f(u,v)\in R^{n-r}$, $g=g(u,v)\in R^r$ and $q=q(u,v)\in R^r$ are given smooth solutions of $(u,v)\in G$. 
Under the modified GKC, we show the existence of boundary-layers for this nonlinear system and derive its reduced boundary condition. We will only consider smooth solutions near the corner $(x,t)=(0,0)$.

For the nonlinear problem above, the first set of assumptions is the following analogue of those in Subsection \ref{subsection2.1} for linear problems.
\begin{assumption}\label{asp3.1}
\ 
\begin{enumerate}
\item[\rm{(I)}]
The equation \eqref{3.1} satisfies the structural stability condition in the end of Introduction and the relation \eqref{1.2} holds.
Moreover, $q_v=q_v(u,v)$ is invertible for $(u,v)\in G$ and $q(u,v) = 0$ uniquely determines $v$ in term of $u$, say $v = h(u)$.
\item[\rm{(II)}]
The initial data lie in equilibrium and are 
consistent with the boundary condition:
\begin{equation*}
	Q(U(x,0))=0\quad \text{and}\quad B(U(0,0);0)=0.
\end{equation*}
\item[\rm{(III)}]
The boundary $x=0$ is non-characteristic for the equilibrium system \eqref{3.8} below and is characteristic for the relaxation system. Namely, the matrix $f_u+f_vh_u$ is invertible and $F_U(U)$ has zero eigenvalues. The multiplicity $n^o$ of the zero eigenvalue is assumed to be independent of $U\in G$.
\end{enumerate}
\end{assumption}

On Assumption (I) above, we comment as follows. Since $q_v$ is invertible, the matrix $P(U)$ in the structural stability condition can be defined for $U$ in the whole $G$, instead of only in the equilibrium manifold $\mathcal{E}$:  
\begin{align}\label{3.3}
P(U):=\left({\begin{array}{*{20}c}
  \vspace{1.5mm}I_{n-r} & 0\\
                q_v^{-1}q_u & I_r
  \end{array}}\right)=\left({\begin{array}{*{20}c}
  \vspace{1.5mm}I_{n-r} & 0\\
                -h_u & I_r
  \end{array}}\right)\quad \text{for} ~U\in G.
\end{align}
A simple computation shows
\begin{equation*} 
	P(U)Q_U(U)P^{-1}(U)=P(U)\left({\begin{array}{*{20}c}
  \vspace{1.5mm}0 & 0\\
               q_u & q_v
  \end{array}}\right)P^{-1}(U)=\left({\begin{array}{*{20}c}
  \vspace{1.5mm}0 & 0\\
                0 & q_v
  \end{array}}\right)\quad \text{for}~ U\in G.
\end{equation*}
 The structural stability condition implies that there exists a positive definite matrix $A_0(U)$ such that
\begin{eqnarray*} 
A_0(U)F_U(U)=F_U^*(U)A_0(U)\quad \text{for} \ U \in G.
\end{eqnarray*}
Thus we can define, for $U\in G$, 
\begin{eqnarray}\label{3.4}
\left\{{\begin{array}{*{20}l}
    \tilde{A}_0(U)\equiv \left({\begin{array}{*{20}c}
  \vspace{1.5mm}A_{01} & A_{012}\\
                A_{012}^* & A_{02}
  \end{array}}\right)(U):=P^{-*}(U)A_0(U)P^{-1}(U),\\[7mm]
    A_1(U)\equiv \left({\begin{array}{*{20}c}
  \vspace{1.5mm}A_{11} & A_{12}\\
                A_{12}^* & A_{22}
  \end{array}}\right)(U):=P^{-*}(U)A_0(U)F_U(U)P^{-1}(U).
  \end{array}}\right. 
\end{eqnarray}
On the other hand, now the equilibrium manifold can be expressed as $\mathcal{E}=\{U\in G: v=h(u)\}$.
By Theorem 2.2 in \cite{Y1}, the block-matrix $A_{012}(U)$ vanishes at $U\in \mathcal{E}$. 
Furthermore, it follows from the stability condition (iii) and the relation \eqref{1.2} that the following matrix is symmetric negative definite for $U\in \mathcal{E}$: 
\begin{align*}
P^{-*}(U)A_0(U)Q_U(U)P^{-1}(U)=\tilde{A}_0(U)P(U)Q_U(U)P^{-1}(U)=&\left({\begin{array}{*{20}c}
\vspace{1.5mm} 0 & \\
                 & A_{02}q_v
  \end{array}}\right)(U).
\end{align*}

Next, we turn to the modified GKC by considering
\begin{eqnarray}\label{3.5}
\left\{
{\begin{array}{*{20}l}
\vspace{1.5mm}
 U_t+F_U(U_0)U_x=Q_U(U_0)U/\epsilon, \\[2mm]
BU=b(t).
  \end{array}}\right.
\end{eqnarray}
This can be viewed as a linearization of \eqref{3.1} with \eqref{3.2} at the corner $(U,t)=(U_0,0)$ when $U_0 = U(0,0)$,
$B=B_U(U_0;0)$, and $b(t)=B_U(U_0;0)U_0-t \partial_t B(U_0;0)$. 
Multiplying \eqref{3.5} with $P^{-*}(U_0)A_0(U_0)$ from the left and setting $V=P(U_0)U$, we obtain 
\begin{eqnarray}\label{3.6}
\left\{
{\begin{array}{*{20}l}
\vspace{2mm}
\left({\begin{array}{*{20}c}
\vspace{1.5mm} A_{01}(U_0) & \\
                 & A_{02}(U_0)
  \end{array}}\right)V_t+A_1(U_0)V_x=\dfrac{1}{\epsilon}\left({\begin{array}{*{20}c}
\vspace{1.5mm} 0 & \\
                 & A_{02}q_v(U_0)
  \end{array}}\right)V, \\[5mm]
BP^{-1}(U_0)V(0,t)=b(t).
  \end{array}}\right.
\end{eqnarray}
This has the form of problem \eqref{2.1} in Section \ref{section2}. Our last assumption reads as

\begin{assumption}\label{asp3.2}
The modified GKC holds for the linearized problem \eqref{3.6}.
\end{assumption}

Under this assumption, Lemma \ref{lemma2.4} holds. Since now the boundary matrix in \eqref{3.6} is $BP^{-1}(U_0)=(B_u+B_vh_u, B_v)(U_0)$, the lemma can be stated as that the matrix 
\begin{equation}\label{3.7} 
	\left({\begin{array}{*{20}c}
  \vspace{2mm} (B_u+B_vh_u)R_1^U & (B_v-(B_u+B_vh_u)A_{11}^{-1}A_{12})G_1\\
                  0 &  L_2^U
  \end{array}}\right)(U_0)
\end{equation}
is invertible. 
Here $R_1^U$ is a right-unstable matrix of $A_{01}^{-1}A_{11}(U_0)$, $L_2^U$ is a left-unstable matrix of $\Lambda_2^{-1}G_2$ with $\Lambda_2$ representing the nonzero eigenvalues of $(A_{22}-A_{12}^*A_{11}^{-1}A_{12})(U_0)$. The matrices $G_1$ and $G_2$ are defined by \eqref{2.38} and \eqref{2.39} with $S=A_{02}q_v(U_0)$.

\subsection{Existence of boundary-layers}\label{subsection3.2}
Following \cite{Y2}, we consider the Ansatz 
\begin{align*}
\left({\begin{array}{*{20}c}
  \vspace{1.5mm}u_\epsilon\\
                v_\epsilon
  \end{array}}\right)(x,t)=\left({\begin{array}{*{20}c}
  \vspace{1.5mm}\bar{u}\\
                \bar{v}
  \end{array}}\right)(x,t)+\left({\begin{array}{*{20}c}
  \vspace{1.5mm}\tilde{u}\\
                \tilde{v}
  \end{array}}\right)(\xi,t)-\left({\begin{array}{*{20}c}
  \vspace{1.5mm}\bar{u}\\
                \bar{v}
  \end{array}}\right)(0,t)
\end{align*}
with $\xi=x/\epsilon$.
The outer solution $(\bar{u},\bar{v})$ satisfies the equilibrium system
\begin{eqnarray}\label{3.8}
\left\{
{\begin{array}{*{20}l}
\vspace{1.2mm}\bar{u}_t+f(\bar{u},h(\bar{u}))_x=0, \\[2mm]
\bar{v}=h(\bar{u}),
  \end{array}}\right.
\end{eqnarray}
while the inner solution $(\tilde{u},\tilde{v})$ solves
\begin{eqnarray}\label{3.9}
\left\{
{\begin{array}{*{20}l}
\vspace{1.2mm}f(\tilde{u},\tilde{v})_{\xi}=0, \\[2mm]
g(\tilde{u},\tilde{v})_{\xi}=q(\tilde{u},\tilde{v}).
  \end{array}}
 \right.
\end{eqnarray}
Based on the matching principle, it was assumed that $(\bar{u}(0,t),\bar{v}(0,t))=(\tilde{u}(\infty,t),\tilde{v}(\infty,t))$.

From the first equation in \eqref{3.9} it follows that 
\begin{equation*}
  f(\tilde{u}(\xi,t),\tilde{v}(\xi,t))=f(\tilde{u}(\infty,t),\tilde{v}(\infty,t)),\quad \forall ~\xi.
\end{equation*}
Set $\tilde{u}_{\infty}=\tilde{u}(\infty,t)$, $\mu:=\tilde{u}(\xi,t)-\tilde{u}_{\infty}$ and $\nu:=\tilde{v}(\xi,t)-h(\mu+\tilde{u}_{\infty})$. Then we have 
\begin{equation}\label{3.10}
  \varphi_1(\mu,\nu,\tilde{u}_{\infty}):=f(\mu+\tilde{u}_{\infty},\nu+h(\mu+\tilde{u}_{\infty}))-f(\tilde{u}_{\infty},h(\tilde{u}_{\infty}))\equiv 0 
\end{equation}
due to the relation $q(\tilde{u}(\infty,t),\tilde{v}(\infty,t))=0$. About this equation, we have
\begin{prop}\label{prop3.1}
There exists a unique function $\psi_1=\psi_1(\nu,\tilde{u}_{\infty})$ defined in a neighborhood of $(0,u_o)\equiv (0,u(0,0))$, such that $\mu=\psi_1(\nu,\tilde{u}_{\infty})$ solves \eqref{3.10} and satisfies 
\begin{equation}\label{3.12}
\left.\frac{\partial \psi_1}{\partial \nu}\right|_{(0,u_o)}=-[(f_u+f_vh_u)^{-1}f_v](u_o,h(u_o)),\qquad
\left.\frac{\partial \psi_1}{\partial \tilde{u}_{\infty}}\right|_{(0,u_o)}= 0.  
\end{equation}
\end{prop}
\begin{proof}
A simple computation shows 
\begin{align*}
\left.\frac{\partial \varphi_1}{\partial \tilde{u}_{\infty}}\right|_{(0,0,u_o)}=0,\quad 
\left.\frac{\partial \varphi_1}{\partial \nu}\right|_{(0,0,u_o)}=f_v(u_o,h(u_o)),\quad 
\left.\frac{\partial \varphi_1}{\partial \mu}\right|_{(0,0,u_o)}=[f_u+f_vh_u](u_o,h(u_o)).
\end{align*}
Since $[f_u+f_vh_u](u_o,h(u_o))$ is invertible, we can apply the implicit function theorem to $\varphi_1=\varphi_1(\mu,\nu,\tilde{u}_{\infty})$ at $(0,0,u_o)$ to get the function $\psi_1(\nu,\tilde{u}_{\infty})$ such that 
$\mu=\psi_1(\nu,\tilde{u}_{\infty})$. Thus the relations in \eqref{3.12} immediately follow.
\end{proof}

On the other hand, in terms of the new variables $\mu$ and $\nu$, the equation \eqref{3.9} can be rewritten as
\begin{eqnarray}\label{3.13}
\left\{
{\begin{array}{*{20}l}
\vspace{1.2mm}(f_u+f_vh_u)\mu'+f_v\nu'=0,\\[2mm]
(g_u+g_vh_u)\mu'+g_v\nu'=\hat{S}\nu,
  \end{array}}
 \right.
\end{eqnarray}
where $f_u, f_v, g_u, g_v, h_u$ are all evaluated at $U_{\nu} := (\mu+\tilde{u}_{\infty},\nu+h(\mu+\tilde{u}_{\infty}))$ and 
$$
\hat{S}=\int_0^1q_v(\mu+\tilde{u}_{\infty},\theta \nu+h(\mu+\tilde{u}_{\infty}))d\theta.
$$ 
Multiplying \eqref{3.13} with $P^{-*}A_0=P^{-*}A_0(U_{\nu})$ from the left yields
$$
P^{-*}A_0F_UP^{-1}\left({\begin{array}{*{20}c}
  \vspace{1.5mm}\mu'\\
                \nu'
  \end{array}}\right)=P^{-*}A_0P^{-1}P\left({\begin{array}{*{20}c}
  \vspace{1.5mm}0 &  \\
                  & \hat{S}
  \end{array}}\right)\left({\begin{array}{*{20}c}
  \vspace{1.5mm}\mu\\
                \nu
  \end{array}}\right).
$$ 
According to \eqref{3.4}, this can be written as 
\begin{align*}
 \left({\begin{array}{*{20}c}
  \vspace{1.5mm}A_{11} & A_{12}\\
                A_{12}^* & A_{22}
  \end{array}}\right)\left({\begin{array}{*{20}c}
  \vspace{1.5mm}\mu'\\
                \nu'
  \end{array}}\right)=\left({\begin{array}{*{20}c}
  \vspace{1.5mm}0 & A_{012}\hat{S}\\
                0 & A_{02}\hat{S}
  \end{array}}\right)\left({\begin{array}{*{20}c}
  \vspace{1.5mm}\mu\\
                \nu
  \end{array}}\right).
\end{align*}
Here $A_{11}, A_{12}, A_{22}, A_{012}$ and $A_{02}$ are all evaluated at $U_{\nu}$. 
Notice that the matrix $A_{11}(U_{\nu})$ is invertible for $U_{\nu}$ close to $U_0$ since 
the matrix $A_{11}(U_0)=[A_{01}(f_u+f_vh_u)](U_0)$ is invertible due to Assumption \ref{asp3.1} (III).
Then we deduce from the last equation that 
\begin{equation*}
	\mu'=-A_{11}^{-1}A_{12}\nu'+A_{11}^{-1}A_{012}\hat{S}\nu
\end{equation*}
and  
\begin{equation}\label{3.14}
(A_{22}-A_{12}^*A_{11}^{-1}A_{12})\nu'=(A_{02}\hat{S}-A_{12}^*A_{11}^{-1}A_{012}\hat{S})\nu.
\end{equation}

Furthermore, from Assumption \ref{asp3.1} (III) and the congruent transformation 
\begin{align*} 
\left({\begin{array}{*{20}c}
  \vspace{2mm}I & 0 \\
  -A_{12}^*A_{11}^{-1} & I
  \end{array}}\right)
    \left({\begin{array}{*{20}c}
  \vspace{2mm}A_{11} & A_{12} \\
  A_{12}^* & A_{22}
  \end{array}}\right)
  \left({\begin{array}{*{20}c}
  \vspace{2mm}I & -A_{11}^{-1}A_{12} \\
  0 & I
  \end{array}}\right) 
  =\left({\begin{array}{*{20}c}
  \vspace{2mm}A_{11} & 0 \\
  0 & A_{22}-A_{12}^*A_{11}^{-1}A_{12}
  \end{array}}\right), 
\end{align*}
we see that the symmetric matrix $(A_{22}-A_{12}^*A_{11}^{-1}A_{12})(U_{\nu})$ in \eqref{3.14} has precisely $n^o$ zero eigenvalues for $U_{\nu}\in G$.
Thus, there exists an $r\times n^o$-matrix $P_0=P_0(U_{\nu})$ such that 
$$
P_0^*(A_{22}-A_{12}^*A_{11}^{-1}A_{12})\equiv 0 
$$ 
and there is an $r\times (r-n^o)$-matrix $P_2=P_2(U_{\nu})$ such that 
$P_2^*(A_{22}-A_{12}^*A_{11}^{-1}A_{12})(U_\nu)P_2$ is invertible.
In addition, $P_0$ and $P_2$ are smooth with respect to $U_{\nu}$ \cite{Ka} and have properties $P_0^*P_0=I_{n^o}$, $P_2^*P_2=I_{r-n^o}$ and 
$P_2^*P_0=0$.

Multiplying \eqref{3.14} with $P_0^*$ from the left gives
\begin{equation*} 
  P_0^*(A_{02}\hat{S}-A_{12}^*A_{11}^{-1}A_{012}\hat{S})\nu=0.
\end{equation*}
Let $\bar{P}_0=P_0(U_0)$ and $\bar{P}_2=P_2(U_0)$ be constant matrices. We decompose $\nu=\bar{P}_2w_2 + \bar{P}_0w_0$ with $w_2=\bar{P}_2^*\nu$ and $w_0=\bar{P}_0^*\nu$. Then the last equation can be rewritten as
\begin{align}\label{3.16}
  \varphi_2(w_2,w_0,\tilde{u}_{\infty}):=P_0^*(A_{02}\hat{S}-A_{12}^*A_{11}^{-1}A_{012}\hat{S})(\bar{P}_2w_2+\bar{P}_0w_0)\equiv0.
\end{align}
About this equation, we have
\begin{prop}\label{prop3.2}
There exists a unique function $\psi_2=\psi_2(w_2,\tilde{u}_{\infty})$ defined in a neighborhood of $(0,u_o)\equiv (0,u(0,0))$, such that $w_0=\psi_2(w_2,\tilde{u}_{\infty})$ solves \eqref{3.16} and satisfies 
\begin{equation}\label{3.18}
  \left.\dfrac{\partial \psi_2}{\partial w_2}\right|_{(0,u_o)}=-(\bar{P}_0^*A_{02}q_v\bar{P}_0)^{-1}(\bar{P}_0^*A_{02}q_v\bar{P}_2)\bigg|_{U=U_0},\qquad 
  \left.\dfrac{\partial \psi_2}{\partial \tilde{u}_{\infty}}\right|_{(0,u_o)}=0.
\end{equation} 
\end{prop}
\begin{proof}
Since $A_{012}(U_0)=0$, it follows from \eqref{3.16} that
\begin{align*}
\left.\dfrac{\partial \varphi_2}{\partial w_0}\right|_{(0,0,u_o)}
= P_0^*(A_{02}\hat{S}-A_{12}^*A_{11}^{-1}A_{012}\hat{S}) \bar{P}_0\bigg|_{U=U_0}
=\bar{P}_0^*A_{02}q_v(U_0)\bar{P}_0 
\end{align*}
is negative definite. We can apply the implicit function theorem to $\varphi_2=\varphi_2(w_2,w_0,\tilde{u}_{\infty})$ at $(0,0,u_o)$ to get a function $\psi_2(w_2,\tilde{u}_{\infty})$ in a neighborhood of $(0,u_o)$ such that $w_0=\psi_2(w_2,\tilde{u}_{\infty})$. Moreover, the computations
$$
\left.\dfrac{\partial\varphi_2}{\partial \tilde{u}_{\infty}}\right|_{(0,0,u_o)}=0 
$$
and
$$
\left.\dfrac{\partial\varphi_2}{\partial w_2}\right|_{(0,0,u_o)}=P_0^*(A_{02}\hat{S}-A_{12}^*A_{11}^{-1}A_{012}\hat{S}) \bar{P}_2\bigg|_{U=U_0}=\bar{P}_0^*A_{02}q_v(U_0)\bar{P}_2 
$$
imply the relation \eqref{3.18}.
\end{proof}

Next we multiply \eqref{3.14} with $\bar{P}_2^*$ from the left to obtain
\begin{align*} 
\bar{P}_2^*(A_{22}-A_{12}^*A_{11}^{-1}A_{12})(\bar{P}_2w_2'+\bar{P}_0w_0')=\bar{P}_2^*(A_{02}\hat{S}-A_{12}^*A_{11}^{-1}A_{012}\hat{S})\nu.
\end{align*}
By Proposition \ref{prop3.2}, we can express $w_0$ in terms of $w_2$ and rewrite the last equation as 
\begin{align*}
K(w_2,\tilde{u}_{\infty})w_2'=\bar{P}_2^*(A_{02}\hat{S}-A_{12}^*A_{11}^{-1}A_{012}\hat{S})[\bar{P}_2w_2+\bar{P}_0\psi_2(w_2,\tilde{u}_{\infty})],
\end{align*}
where 
$$
K(w_2,\tilde{u}_{\infty})=\bar{P}_2^*(A_{22}-A_{12}^*A_{11}^{-1}A_{12})\left[\bar{P}_2+\bar{P}_0\dfrac{\partial \psi_2}{\partial w_2}\right].
$$ 
Notice that $K(0,u_o)=\bar{P}_2^*(A_{22}-A_{12}^*A_{11}^{-1}A_{12})(U_0)\bar{P}_2= \Lambda_2$ is invertible and thereby $K(w_2,\tilde{u}_{\infty})$ is invertible near $(0,u_o)$. Then the last equation can be rewritten as
\begin{align}
w_2'= &~K^{-1}(w_2,\tilde{u}_{\infty})\bar{P}_2^*(A_{02}\hat{S}-A_{12}^*A_{11}^{-1}A_{012}\hat{S})[\bar{P}_2w_2+\bar{P}_0\psi_2(w_2,\tilde{u}_{\infty})]\nonumber\\[2mm]
\equiv &~H(w_2,\tilde{u}_{\infty}).\label{3.19}
\end{align}

By Proposition \ref{prop3.2}, we have $\psi_2(0,u_o)=0$ and thereby $H(0,u_o)=0$. Namely, $(w_2,\tilde{u}_{\infty})=(0,u_o)$ is a critical point for the dynamical system \eqref{3.19} with parameter $\tilde{u}_{\infty}$. 
Moreover, we use the relation \eqref{3.18} and compute
\begin{align*}
\left.\dfrac{\partial H}{\partial w_2}\right|_{(0,u_o)}
=&K^{-1}(0,u_o)\bigg[\bar{P}_2^*A_{02}q_v\bar{P}_2-(\bar{P}_2^*A_{02}q_v\bar{P}_0)(\bar{P}_0^* A_{02}q_v \bar{P}_0)^{-1}(\bar{P}_0^* A_{02}q_v \bar{P}_2)\bigg](U_0)\\[2mm]
=&\Lambda_2^{-1}G_2,
\end{align*}
where $G_2$ is defined in \eqref{2.39} with $S=A_{02}q_v$.
By referring to the discussion below \eqref{2.44}, this Jacobian matrix has $(n^+-n_1^+)$ negative eigenvalues and $(r-n^o-n^++n_1^+)$ positive eigenvalues. 
According to the theory of ordinary differential equations \cite{Co}, there is a $(n^+-n_1^+)$-dimensional stable manifold $\mathcal{S}(\tilde{u}_{\infty})$ of $H(w_2,\tilde{u}_{\infty})$ near $w_2=0$. Then there are $(r-n^o-n^++n_1^+)$ functions $\tilde{S}_i(w_2,\tilde{u}_{\infty})$, defined in a neighborhood of $(0,u_o)$, such that 
$$
\mathcal{S}(\tilde{u}_{\infty})=\{w_2\in R^{r-n^o}\ |\ \tilde{S}_i(w_2,\tilde{u}_{\infty})=0,\ i=1,2,...,(r-n^o-n^++n_1^+) \}.
$$
Moreover, $\tilde{S}:=(\tilde{S}_1,\tilde{S}_2,...\tilde{S}_{r-n^o-n^++n_1^+})^*$ has the following properties
\begin{equation}\label{3.20}
  \tilde{S}(0,\tilde{u}_{\infty})\equiv0,\quad  \dfrac{\partial \tilde{S}}{\partial w_2}(0,\tilde{u}_{\infty})=L_2^U(\tilde{u}_{\infty}),
\end{equation}
where $L_2^U(\tilde{u}_{\infty})$ is a left-unstable matrix of $\frac{\partial H}{\partial w_2}(0,\tilde{u}_{\infty})$.

With $\tilde{S}=\tilde{S}(w_2,\tilde{u}_{\infty})$ defined above, we turn to the boundary condition 
$$
B(\mu+\tilde{u}_{\infty},\nu+h(\mu+\tilde{u}_{\infty});t)=0
$$
and find the initial value $w_2(0)$ for the dynamical system \eqref{3.19}. 
For this purpose, we firstly recall the decomposition $\tilde{u}_{\infty}=\bar{u}(0,t)=R_1^U \alpha + R_1^S \beta$ in \eqref{2.46}, where $R_1^S$ and $R_1^U$ are right-stable and right-unstable matrix of $[f_u+f_vh_u](u_o,h(u_o)) $, respectively. 
On the other hand, we know from Propositions \ref{prop3.1} and \ref{prop3.2} that $\mu=\psi_1(\nu,\tilde{u}_{\infty})$ and $w_0=\psi_2(w_2,\tilde{u}_{\infty})$ while $\nu=\bar{P}_2w_2+\bar{P}_0w_0$ and $w_2=w_2(0)$.
With these, we define 
\begin{align*}
\Psi(\beta,\alpha ,w_2,t):= B(\psi_1(\nu,\tilde{u}_{\infty})+\tilde{u}_{\infty},\nu+h(\psi_1(\nu,\tilde{u}_{\infty})+\tilde{u}_{\infty}),t).
\end{align*}
In this way, we arrive at the following system of $(r-n^o+n_1^+)$ nonlinear algebraic equations:
\begin{equation}\label{temp1}
\left\{
{\begin{array}{*{20}l}
\vspace{1.2mm}\Psi(\beta,\alpha ,w_2,t) = 0  ,\\[2mm]
\tilde{S}(w_2,R_1^S \beta+ R_1^U \alpha) = 0.
  \end{array}}
 \right.  
\end{equation}

To solve these equations, we try to employ the implicit function theorem. To do this, we use 
\eqref{3.12} and \eqref{3.18}:
$$
\left.\dfrac{\partial \psi_1}{\partial \tilde{u}_{\infty}}\right|_{(0,u_o)}=0,\quad
\left.\dfrac{\partial \psi_2}{\partial \tilde{u}_{\infty}}\right|_{(0,u_o)}=0\quad\text{and thereby}\ 
\left.\dfrac{\partial \nu}{\partial \tilde{u}_{\infty}}\right|_{(0,u_o)}=0
$$ 
to compute 
\begin{align}\label{3.21}
\left.\dfrac{\partial \Psi}{\partial \alpha}\right|_{(\beta_o,\alpha_o,0,0)}
=&(B_u+B_vh_u)R_1^U\bigg|_{U=U_0}.
\end{align}
Here $\alpha_o$ and $\beta_o$ denote the values of $\alpha$ and $\beta$ at $t=0$.
Furthermore, we deduce from \eqref{3.12} and \eqref{3.18} that
\begin{align*}
\left.\dfrac{\partial \Psi}{\partial w_2}\right|_{(\beta_o,\alpha_o,0,0)}
=&\left[(B_u+B_vh_u)\frac{\partial \psi_1}{\partial \nu}+B_v\right]\left[\bar{P}_2+\bar{P}_0\frac{\partial \psi_2}{\partial w_2}\right]\bigg|_{(\beta_o,\alpha_o,0,0)} \\[2mm]
=&~[B_v-(B_u+B_vh_u)(f_u+f_vh_u)^{-1}f_v]G_1\bigg|_{U=U_0} 
\end{align*}
with $G_1=[\bar{P}_2-\bar{P}_0(\bar{P}_0^*A_{02}q_v\bar{P}_0)^{-1}(\bar{P}_0^*A_{02}q_v\bar{P}_2)](U_0)$.
Combining this with \eqref{3.20} and \eqref{3.21}, we have
$$
\left.\dfrac{\partial (\Psi,\tilde{S})}{\partial (\alpha,w_2)}\right|_{(\beta_o,\alpha_o,0,0)}
=\left(
{\begin{array}{*{20}c}
\vspace{1.2mm}(B_u+B_vh_u)R_1^U & [B_v-(B_u+B_vh_u)(f_u+f_vh_u)^{-1}f_v]G_1 ,\\[2mm]
                   0            & L_2^U
  \end{array}}
 \right)\Bigg|_{U=U_0}.
$$
This matrix is invertible due to $[(f_u+f_vh_u)^{-1}f_v](u_o,h(u_o))=A_{11}^{-1}A_{12}(U_0)$ and the modified GKC. Thus we can apply the implicit function theorem to \eqref{temp1} and obtain the following solution 
$$
\alpha=\tilde{\alpha}(\beta,t),\qquad \quad w_2=\tilde{w_2}(\beta,t)
$$
locally. 
Like that for the linear problem, the first relation is the reduced boundary condition for the equilibrium system. 

Once the reduced problem is solved, the outgoing mode $\beta$ is known. With this $\beta$, the initial value $w_2(0)$ can be determined by the second relation above. 
With $w_2(0)$ thus obtained, we solve the dynamical system \eqref{3.19} to get $w_2=w_2(\xi,t)$. In this way, the boundary-layer $(\mu,\nu)$ is obtained by using Propositions \ref{prop3.1} and \ref{prop3.2}.
 
\

\begin{appendices}
\section{Appendix}
\subsection{Some details in the proof of Lemma \ref{lemma2.1}}\label{AppendA}
In this appendix, we show that the matrices $S_{00}(\xi,\eta)$ and $I_{n-r}+N_1(\xi,\eta)$ are invertible for any $\eta\geq 0$ and any $\xi$ with $Re\xi>0$, which are used in the proof of Lemma \ref{lemma2.1}. 

Notice that $-S_{00}=-P_0^*(\eta S-\xi A_{02})P_0$ can be expressed in the form $A+iB$ with $A$ a symmetric positive definite matrix and $B$ a symmetric matrix. In order to show that $S_{00}$ is invertible, it suffices to prove 
\begin{prop}\label{lemmaA.1}
Let $A$ be a symmetric positive definite matrix and $B$ be a symmetric matrix. Then the complex matrix $A+iB$ is invertible.
\end{prop}
\begin{proof}
By our assumptions, the matrix $A^{-1/2}BA^{-1/2}$ is symmetric. Then there exists an orthonormal matrix $O$ such that $O^*A^{-1/2}BA^{-1/2}O=\Lambda$ with $\Lambda$ a real diagonal matrix. Thus
\begin{align*}
A+iB=&A^{1/2}(I+iA^{-1/2}BA^{-1/2})A^{1/2}\\[2mm]
    =&A^{1/2}(I+iO \Lambda O^*)A^{1/2}\\[2mm]
    =&A^{1/2}O(I+i \Lambda )O^*A^{1/2}.
\end{align*}
Clearly, the diagonal matrix $I+i \Lambda$ is invertible and hence $A+iB$ is invertible. \\
\end{proof}

From the proof we see that  
\begin{align*}
(A+iB)^{-1}=A^{-1/2}O(I+i \Lambda )^{-1}O^*A^{-1/2}.
\end{align*}
Clearly, $(I+i \Lambda )^{-1}=\Lambda_1+i \Lambda_2$ where $\Lambda_1$, $\Lambda_2$ are real diagonal matrices and the elements of $\Lambda_1$ are positive. Therefore, the inverse matrix $(A+iB)^{-1}$ can also be expressed as $\bar{A}+i\bar{B}$ with $\bar{A}$ a symmetric positive definite matrix and $\bar{B}$ a symmetric matrix. 

As a corollary, we see that the inverse matrix 
$$
-\xi S_{00}^{-1}=-\big(P_0^*(\frac{\eta}{\xi} S-  A_{02})P_0\big)^{-1}
$$ 
can be written as $A+iB$ with $A$ a symmetric positive definite matrix and $B$ a symmetric matrix. Then the matrix
\begin{align*}
A_{01}^{1/2}\big(I_{n-r}+N_1(\xi,\eta)\big)A_{01}^{-1/2}=&I_{n-r}+A_{01}^{1/2}(A_{11}^{-1}A_{12}P_0) (A+iB) (P_0^*A_{12}^*A_{11}^{-1})A_{01}^{1/2}
\end{align*}
also satisfies the condition in Proposition \ref{lemmaA.1}. Consequently, the matrix $I_{n-r}+N_1(\xi,\eta)$ is invertible.

\subsection{Proof of Lemma \ref{lemma2.2}}\label{AppendB}
In this part, we present a proof of Lemma \ref{lemma2.2}.
\begin{proof}
Let $\lambda$ be an eigenvalue of $M(\xi,\eta)$. Then we have $|\lambda I- M(\xi,\eta)|=0$
and thereby 
\begin{equation}\label{B.1}
  \bigg|\lambda\left({\begin{array}{*{20}c}
  \vspace{1.5mm}I_{n-r}+N_1 & N_2\\
  0 & I_{r-n^o}
  \end{array}}\right)-\left({\begin{array}{*{20}c}
  \vspace{1.5mm}-\xi A_{11}^{-1}A_{01} & 0\\
  N_3 & N_4
  \end{array}}\right)\bigg|=0 
\end{equation}
due to \eqref{2.23}.
Recall from \eqref{2.22} that 
$$
N_1 = -\xi (A_{11}^{-1}A_{12}P_0) S_{00}^{-1} (P_0^*A_{12}^*A_{11}^{-1})A_{01}.
$$
Set $\tilde{P}_0:=A_{11}^{-1}A_{12}P_0$ 
and
\begin{eqnarray*} 
\left\{{\begin{array}{*{20}l}
\vspace{3.5mm}B_{11}=&\lambda\left({\begin{array}{*{20}c}
  \vspace{1.5mm}I_{n-r} & N_2\\
  0 & I_{r-n^o}
  \end{array}}\right)-\left({\begin{array}{*{20}c}
  \vspace{1.5mm}-\xi A_{11}^{-1}A_{01} & 0\\
  N_3 & N_4
  \end{array}}\right),\\[3mm]
\vspace{3mm}B_{12}=&\left({\begin{array}{*{20}c}
  \vspace{1.5mm}\lambda \tilde{P}_0 \\
           0           
  \end{array}}\right),\\[3mm]
\vspace{2mm}B_{21}=&\left({\begin{array}{*{20}c}
  -\xi \tilde{P}_0^* A_{01},& 0                  
  \end{array}}\right),\\[2mm]
B_{22}=&-S_{00}.
\end{array}}\right.
\end{eqnarray*}
Then \eqref{B.1} actually is
$$\big|B_{11}-B_{12}B_{22}^{-1}B_{21}\big|=0,$$
or 
\begin{align*}
0=&\big|B_{11}-B_{12}B_{22}^{-1}B_{21}\big|\big|B_{22}\big|
 =\bigg|\left({\begin{array}{*{20}c}
  \vspace{1.5mm}B_{11} & B_{12} \\
  B_{21} & B_{22}
  \end{array}}\right) \bigg|\\[2mm]
 =&\left|\left({\begin{array}{*{20}c}
  \vspace{1.5mm}\lambda I & \lambda N_2 & \lambda \tilde{P}_0 \\
  \vspace{1.5mm}0 & \lambda I & 0\\
  0 & 0 & 0
  \end{array}}\right)-
  \left({\begin{array}{*{20}c}
  \vspace{1.5mm}-\xi A_{11}^{-1}A_{01} & 0 & 0 \\
  \vspace{1.5mm}N_3 & N_4 & 0\\
  \xi \tilde{P}_0^* A_{01} & 0 & S_{00}
  \end{array}}\right) \right|.
\end{align*}
Multiplying the last determinant with
$$\left|\left({\begin{array}{*{20}c}
  \vspace{1.5mm} I & - N_2 & - \tilde{P}_0 \\
  \vspace{1.5mm}0 &  I & 0\\
  0 & 0 & I
  \end{array}}\right) \right| $$
from the right, we obtain
\begin{align*}
  0=&\left|\lambda\left({\begin{array}{*{20}c}
  \vspace{1.5mm} I &  & \\
  \vspace{1.5mm} &  I & \\
      &  & 0
  \end{array}}\right)
  -\left({\begin{array}{*{20}c}
  \vspace{1.5mm}-\xi A_{11}^{-1}A_{01} & \xi A_{11}^{-1}A_{01} N_2 & \xi A_{11}^{-1}A_{01}\tilde{P}_0 \\[2mm]
  \vspace{1.5mm}N_3 & N_4-N_3N_2 & -N_3\tilde{P}_0 \\[2mm]
  \vspace{1.7mm}         \xi \tilde{P}_0^* A_{01} & -\xi \tilde{P}_0^* A_{01}N_2  & S_{00}-\xi \tilde{P}_0^*A_{01}\tilde{P}_0 
  \end{array}}\right)\right|.
\end{align*}

Recall from \eqref{2.22} that  
\begin{align*}
N_2=\tilde{P}_2-\tilde{P}_0S_{00}^{-1}S_{02}, \quad 
N_3=\xi \Lambda_2^{-1} (\tilde{P}_2^*-S_{20}S_{00}^{-1}\tilde{P}_0^*)A_{01},
\end{align*}
where $\tilde{P}_2:=A_{11}^{-1}A_{12}P_2$. Then we multiply the last determinant with
$$\left|\left({\begin{array}{*{20}c}
  \vspace{1.5mm}A_{11} &  & \\
  \vspace{1.5mm} & \Lambda_2 & \\
      &  & I
  \end{array}}\right)\right| $$
from the left to get
\begin{align*}
  0=&\Bigg|\lambda
  \left({\begin{array}{*{20}c}
  \vspace{1.5mm}A_{11} &  & \\
  \vspace{1.5mm} & \Lambda_2 & \\
      &  & 0
  \end{array}}\right)
  -\left({\begin{array}{*{20}c}
  \vspace{1.5mm}0   &  & \\
  \vspace{1.5mm}       & \Lambda_2 N_4 & 0\\
                       & 0 & S_{00}
  \end{array}}\right)\\[2mm]
  &-\xi
  \left({\begin{array}{*{20}c}
  \vspace{1.5mm}-A_{01} & A_{01}N_2 & A_{01}\tilde{P}_0 \\[2mm]
  \vspace{1.5mm}(\tilde{P}_2^*-S_{20}S_{00}^{-1}\tilde{P}_0^*)A_{01} & -(\tilde{P}_2^*-S_{20}S_{00}^{-1}\tilde{P}_0^*)A_{01}N_2 & -(\tilde{P}_2^*-S_{20}S_{00}^{-1}\tilde{P}_0^*)A_{01}\tilde{P}_0 \\[2mm]
           \tilde{P}_0^*A_{01} & -\tilde{P}_0^*A_{01}N_2  & -\tilde{P}_0^*A_{01}\tilde{P}_0
  \end{array}}\right)
  \Bigg|.
\end{align*}
Moreover, by multiplying this determinant with
$$\bigg|\left({\begin{array}{*{20}c}
  \vspace{1.5mm}I & 0 &0\\
  \vspace{1.5mm}0  & I & S_{20}S_{00}^{-1} \\
  0 & 0 & I
  \end{array}}\right)\bigg|\quad \text{and}\quad \bigg|\left({\begin{array}{*{20}c}
  \vspace{1.5mm}I & 0 &0\\
  \vspace{1.5mm}0  & I & 0\\
  0 & S_{00}^{-1}S_{02} & I
  \end{array}}\right)\bigg|$$
from the left and from the right respectively, we get 
\begin{align*}
  0=\bigg|&\lambda
  \left({\begin{array}{*{20}c}
  \vspace{1.5mm}A_{11} &  & \\
  \vspace{1.5mm} & \Lambda_2 & \\
      &  & 0
  \end{array}}\right)
  -\left({\begin{array}{*{20}c}
  \vspace{1.5mm}0 &  & \\
  \vspace{1.5mm}       & S_{22} & S_{20}\\
                       & S_{02} & S_{00}
  \end{array}}\right)
  -\xi\left({\begin{array}{*{20}c}
  \vspace{1.5mm}-A_{01} & A_{01}\tilde{P}_2 & A_{01}\tilde{P}_0 \\[2mm]
  \vspace{1.5mm}\tilde{P}_2^*A_{01} & -\tilde{P}_2^*A_{01}\tilde{P}_2 & -\tilde{P}_2^*A_{01}\tilde{P}_0\\[2mm]
           \tilde{P}_0^*A_{01} & -\tilde{P}_0^*A_{01}\tilde{P}_2  & -\tilde{P}_0^*A_{01}\tilde{P}_0
  \end{array}}\right)\bigg|. 
\end{align*}
Finally, we multiply the last determinant with 
  $$\bigg|\left({\begin{array}{*{20}c}
  \vspace{1.5mm}I & 0 & 0\\
  0 & P_2 & P_0
  \end{array}}\right)\bigg|\quad \text{and} \quad \bigg|\left({\begin{array}{*{20}c}
  \vspace{1.5mm}I & 0 \\
  \vspace{1.5mm}0 & P_2^* \\
  0 & P_0^*
  \end{array}}\right)\bigg| $$
from the left and from the right to obtain 
\begin{align*}
  0=&\bigg|\lambda
  \left({\begin{array}{*{20}c}
  \vspace{1.5mm}A_{11} & 0 \\
  0 & A_{22}-A_{12}^*A_{11}^{-1}A_{12}
  \end{array}}\right)
  -\eta
  \left({\begin{array}{*{20}c}
  \vspace{1.5mm}0 & 0 \\
  0 & S
  \end{array}}\right)+\xi 
  \left({\begin{array}{*{20}c}
  \vspace{1.5mm}0 & 0 \\
  0 & A_{02}
  \end{array}}\right)\\[1mm]
  &+\xi\left({\begin{array}{*{20}c}
  \vspace{1.5mm}A_{01} & -A_{01}A_{11}^{-1}A_{12} \\
  -A_{12}^*A_{11}^{-1}A_{01} & A_{12}^*A_{11}^{-1}A_{01}A_{11}^{-1}A_{12}
  \end{array}}\right)
  \bigg|\\[2mm]
  =&\bigg|\left({\begin{array}{*{20}c}
  \vspace{1.5mm}I & 0 \\
  -A_{12}^*A_{11}^{-1} & I
  \end{array}}\right)
  \bigg[\xi \left({\begin{array}{*{20}c}
  \vspace{1.5mm}A_{01} & 0 \\
  0 & A_{02}
  \end{array}}\right)+\lambda 
  \left({\begin{array}{*{20}c}
  \vspace{1.5mm}A_{11} & A_{12} \\
  A_{12}^* & A_{22}
  \end{array}}\right)-\eta
  \left({\begin{array}{*{20}c}
  \vspace{1.5mm}0 & 0\\
  0 & S
  \end{array}}\right)\bigg]
  \left({\begin{array}{*{20}c}
  \vspace{1.5mm}I & -A_{11}^{-1}A_{12} \\
  0 & I
  \end{array}}\right)\bigg|.
\end{align*}
Consequently, we have 
$$\bigg|\xi \left({\begin{array}{*{20}c}
  \vspace{1.5mm}A_{01} & 0 \\
  0 & A_{02}
  \end{array}}\right)+\lambda 
  \left({\begin{array}{*{20}c}
  \vspace{1.5mm}A_{11} & A_{12} \\
  A_{12}^* & A_{22}
  \end{array}}\right)-\eta
  \left({\begin{array}{*{20}c}
  \vspace{1.5mm}0 & 0\\
  0 & S
  \end{array}}\right)\bigg|=0.$$

Now if the eigenvalue $\lambda$ is purely imaginary, the last matrix can be written in form of $A+iB$ with $A$ a symmetric positive definite matrix and $B$ a symmetric matrix. Then the last equation contradicts Proposition \ref{lemmaA.1}. This completes the proof of Lemma \ref{lemma2.2}.\\
\end{proof}

\subsection{A detail in Subsection \ref{subsection2.4}}\label{appendC}
In this appendix, we show the fact used in Subsection \ref{subsection2.4} that the matrix $E(\xi',\eta')$ is uniformly bounded for any $\eta'\geq 0$ and $\xi'$ with $Re\xi'>0$. Recall 
$$
E(\xi',\eta')=\left(E_1(\xi',\eta'), \ E_2(\xi',\eta')\right) =\left(-\xi' S_{00}^{-1} P_0^*A_{12}^*A_{11}^{-1}A_{01},\ \ -S_{00}^{-1}S_{02} \right) 
$$
defined in \eqref{2.25}, where 
$$
S_{00}(\xi',\eta')=\eta'(P_0^*SP_0)-\xi'(P_0^*A_{02}P_0),\quad S_{02}(\xi',\eta')=\eta'(P_0^*SP_2)-\xi'(P_0^*A_{02}P_2) 
$$
are given in \eqref{2.16}. Because $P_0^*A_{12}^*A_{11}^{-1}A_{01}$ is constant, we only need to show that $\xi'S_{00}^{-1}$ and $S_{00}^{-1}S_{02}$ are uniformly bounded. 
When $\eta'=0$, these two matrices are independent of $\xi'$ and thereby uniformly bounded. 

For $\eta'\neq 0$, we have
\begin{align*}
\xi'S_{00}^{-1}=&\left(\frac{\eta'}{\xi'}(P_0^*SP_0)-(P_0^*A_{02}P_0)\right)^{-1},\\[1mm]
S_{00}^{-1}S_{02}=&\left((P_0^*SP_0)-\frac{\xi'}{\eta'}(P_0^*A_{02}P_0)\right)^{-1}P_0^*SP_2-\left(\frac{\eta'}{\xi'}(P_0^*SP_0)-(P_0^*A_{02}P_0)\right)^{-1}P_0^*A_{02}P_2.
\end{align*}
Notice that $-P_0^*SP_0$ and $P_0^*A_{02}P_0$ are two symmetric positive definite matrices. 
Thus, it suffices to prove 
\begin{prop}\label{lemmaD.1}
Suppose $A$ and $B$ are symmetric positive definite matrices with constant coefficients, $a$ and $b$ are real numbers. Then the matrix 
\begin{equation*} 
	(A+aB+ibB)^{-1}
\end{equation*}
is uniformly bounded with respect to any $a\geq 0$ and $b\in (-\infty,+\infty)$.
\end{prop} 

\begin{proof}
Because of the relation 
$$
A+aB+ibB=A^{1/2}\bigg(I+a(A^{-1/2}BA^{-1/2})+ib(A^{-1/2}BA^{-1/2})\bigg)A^{1/2},
$$
we may assume $A=I$. Since $I+aB$ is symmetric positive definite for $a>0$, we have
\begin{align*}
(I+aB+ibB)^{-1}=&(I+aB)^{-1/2}\bigg[I+ib(I+aB)^{-1/2}B(I+aB)^{-1/2}\bigg]^{-1}(I+aB)^{-1/2} \\
           \equiv&(I+aB)^{-1/2}(I+i\tilde{B})^{-1}(I+aB)^{-1/2}, 
\end{align*}
where $\tilde{B}$ is again symmetric. 
Let $O$ be an orthonormal matrix such that $O^*\tilde{B}O=\tilde{\Lambda}_B$ with $\tilde{\Lambda}_B$ a real diagonal matrix. Then we have 
$$
I+i\tilde{B} = O(I + i \tilde{\Lambda}_B)O^* 
$$
and thereby the matrix $(I+i\tilde{B})^{-1}$ is uniformly bounded. Similarly, the matrix $(I+aB)^{-1/2}$ is uniformly bounded and hence the proof is complete.
\end{proof}
\end{appendices}




\begin{thebibliography}{99}
\bibitem{BS}
S. Benzoni-Gavage \& D. Serre, Multidimensional hyperbolic partial differential equations: First order systems and applications, Clarendon Press, Oxford, 2007.

\bibitem{BK1}
R. Borsche \& A. Klar, A nonlinear discrete velocity relaxation model for traffic flow, SIAM J.
Appl. Math. 78 (2018), no. 5, 2891–2917.

\bibitem{BK2}
R. Borsche \& A. Klar, Kinetic layers and coupling conditions for scalar equations on networks,
Nonlinearity 31 (2018), no. 7, 3512–3541.

\bibitem{CFL}
Z. Cai \& Y. Fan \& R. Li, A framework on moment model reduction for kinetic equation, SIAM J. Appl. Math. 75:5 (2015), 2001-2023.

\bibitem{CH}
G. Carbou \& B. Hanouzet, Relaxation approximation of some nonlinear Maxwell initial-boundary value problem, Commun. Math. Sci. 4:2 (2006), 331-344.

\bibitem{CS}
D. Chakraborty \& J. E. Sader, Constitutive models for linear compressible viscoelastic flows
of simpleliquids at nanometer length scales, Phys. Fluids 27, 052002 (2015).

\bibitem{CLL}
G.-Q. Chen \& C. Levermore \& T.-P. Liu, Hyperbolic conservation laws with stiff
relaxation terms and entropy, Comm. Pure Appl. Math. 47 (1994), 787-830.

\bibitem{Co}
E. A. Coddington \& N. Levinson, Theory of ordinary differential equations, McGraw-Hill, New York, 1955.

\bibitem{Da}
C. M. Dafermos, Hyperbolic Conservation Laws in Continuum Physics, 3rd ed., Springer-Verlag, Berlin, 2010.

\bibitem{Ga}
R. Gatignol, Théorie Cinétique de Gaz à Répartition Discrète de Vitesses, Springer, New York, 1975.

\bibitem{GM}
V. Giovangigli \& M. Massot, Asymptotic stability of equilibrium states for multicomponent reactive flows, Math. Models Meth. Appl. Sci. 8, 251(1998).

\bibitem{GKO}
B. Gustafsson \& H. O. Kreiss \& J. Oliger, Time-Dependent Problems and Difference
Methods, 2nd ed., John Wiley \& Sons, Inc., Hoboken, New Jersey, 2013.
 
\bibitem{Ha}
B. Hanouzet \& P. Huynh, Approximation par relaxation d’un système de Maxwell non linéaire, C. R. Acad. Sci., Ser. I: Math. 330, 193 (2000).

\bibitem{Hi}
R. L. Higdon, Initial-boundary value problems for linear hyperbolic systems, SIAM
Rev. 28(1986), 177-217.

\bibitem{Ka}
T. Kato, A short introduction to perturbation theory for linear operators, Springer,
New York, 1982.

\bibitem{Kr}
H.-O. Kreiss, Initial boundary value problems for hyperbolic systems, Comm. Pure
Appl. Math. 23 (1970), 277-298.

\bibitem{Le}
C. D. Levermore, Moment closure hierarchies for kinetic theories, J. Statis. Phys. 83 (1996), 1021-1065.

\bibitem{MO}
A. Majda \& S. Osher, Initial-boundary value problems for hyperbolic equations with
uniformly characteristic boundary, Comm. Pure Appl. Math. 28 (1975), 607-675.

\bibitem{Mi}
L. Mieussens, Discrete-velocity models and numerical schemes for the Boltzmann-BGK equation in plane and axisymmetric geometries, J. Comput. Phys. 162 (2), 429-466.

\bibitem{Na}
R. Natalini, Recent results on hyperbolic relaxation problems, Analysis of systems of conservation laws (Aachen, 1997), Chapman \& Hall/CRC, Boca Raton, FL, 128–198, 1999.


\bibitem{WX}
W.-C. Wang \& Z. Xin, Asymptotic limit of initial-boundary value problems for conservation laws with relaxational extensions, Comm. Pure Appl. Math. 51:5 (1998), 505-535. 

\bibitem{XX}
Z. Xin \& W.-Q. Xu, Stiff well-posedness and asymptotic convergence for a class of linear relaxation systems in a quarter plane, J. Differential Equations 167:2 (2000), 388–437.

\bibitem{XX2}
Z. Xin \& W.-Q. Xu, Initial-boundary value problem to systems of conservation laws with relaxation, Quart. Appl, Math. 60 (2002), 251–281.

\bibitem{Xu1}
W.-Q. Xu, Initial boundary value problem for a class of linear relaxation systems in arbitrary space dimensions, J. Differential Equations 183 (2002), 462–496. 

\bibitem{Xu2}
W.-Q. Xu, Boundary conditions and boundary layers for a multi-dimensional relaxation model, J. Differential Equations 197 (2004), 85-117. 


\bibitem{Y1}
W.-A. Yong, Singular perturbations of first-order hyperbolic systems with stiff source terms,
J. Differential Equations 155(1) (1999), 89–132.

\bibitem{Y2}
W.-A. Yong, Boundary conditions for hyperbolic systems with stiff source terms, Indiana Univ.
Math. J. 48(1) (1999), 115–137.

\bibitem{Y3}
W.-A. Yong, Newtonian limit of Maxwell fluid flows, Arch. Rational Mech. Anal. 214:3 (2014),
913–922.

\bibitem{Y4}
W.-A. Yong, Boundary stabilization of hyperbolic balance laws with characteristic
boundaries, Automatica J. IFAC, 101 (2019) 252–257.

\bibitem{Y5}
W.-A. Yong, An interesting class of partial differential equations, J. Math. Phys. 49 (2008) 033503.

\bibitem{Y7}
W.-A. Yong, Basic aspects of hyperbolic relaxation systems. In: Advances in the Theory
of Shock Waves, H. Freist{\"u}hler et al, (ed.), Progress in Nonlinear Differential
Equations and Their Applications 47, Birkh{\"a}user, Boston, 2001, pp. 259–305

\bibitem{Ze}
Y. Zeng, Thermal non-equilibrium flows in three space dimensions, Arch. Rational Mech. Anal. 219 (2016) 27–87.

\bibitem{ZY}
Y. Zhou \& W.-A. Yong, Construction of boundary conditions for hyperbolic relaxation
approximations I: The linearized Suliciu model, Math. Models Meth. Appl. Sci. 30 (2020) 1407–1439.

\bibitem{ZHYY}
Y. Zhu \& L. Hong \& Z. Yang \& W.-A. Yong, Conservation-dissipation formalism of
irreversible thermodynamics, J. Non-Equilib. Thermodyn. 40:2 (2015) 67-74.




\end{thebibliography}
\end{document}